\documentclass[11pt,reqno]{amsart}

\usepackage{floatrow}

\usepackage{amsmath}
\usepackage{amssymb}
\usepackage[left=1in,top=1in,right=1in,bottom=1in]{geometry}
\usepackage{epsfig}
\usepackage[normalem]{ulem}
\usepackage[usenames,dvipsnames]{color}
\usepackage{todonotes}

\usepackage[colorlinks, citecolor=black, linkcolor=black, urlcolor=black]{hyperref}
\usepackage{bm}
\usepackage{algorithm}
\usepackage{algpseudocode}

\newenvironment{varalgorithm}[1]
  {\algorithm\renewcommand{\thealgorithm}{#1}}
  {\endalgorithm}

\usepackage{enumitem}
\newcommand{\subscript}[2]{$#1 _ #2$}

\allowdisplaybreaks

\usepackage{hyperref}
\usepackage{cleveref}
\crefname{subsection}{subsection}{subsections}

\usepackage{tikz}
\usetikzlibrary{backgrounds}
\usetikzlibrary{intersections}

\newtheorem{theorem}{Theorem}[section]
\newtheorem{lemma}[theorem]{Lemma}
\newtheorem{proposition}[theorem]{Proposition}

\newtheorem{claim}[theorem]{Claim}

\newcommand\eps{\varepsilon}
\newcommand{\E}{\mathbb E}

\newcommand{\Bin}{\mathrm{Bin}}

\newcommand{\Nn}{{\mathbb N}}

\newcommand{\scr}{\mathcal}
\newcommand{\mb}{\mathbb}

\newcommand{\ex}{\mathrm{ex}}

\theoremstyle{definition}
\newtheorem{remark}{Remark}

\def\ex{{\mathbb E}}

\def\W{{\mathcal W}}
\def\T{{\mathcal T}}
\newcommand{\YY}{{\mathcal Y}}
\newcommand{\card}[1]{\left|{#1}\right|}
\newcommand{\tp}{{t+1}}
\newcommand{\ip}{{i+1}}
\newcommand{\set}[1]{\left\{#1\right\}}

\title[Building Hamiltonian Cycles in the Semi-Random Graph Process]{Building Hamiltonian Cycles in the Semi-Random Graph Process in Less Than $2n$ Rounds}

\author{Alan Frieze}
\address{Department of Mathematical Sciences, Carnegie Mellon University, Pittsburgh PA 15213, USA}
\email{alan@random.math.cmu.edu}

\author{Pu Gao}
\address{Department of Combinatorics and Optimization, University of Waterloo, Waterloo, Canada}
\email{pu.gao@uwaterloo.ca}

\author{Calum MacRury}
\address{Graduate School of Business, Columbia University, New York City, USA}
\email{cm4379@columbia.edu}

\author{Pawe\l{} Pra\l{}at}
\address{Department of Mathematics, Toronto Metropolitan University, Toronto, Canada}
\email{pralat@torontomu.ca}

\author{Gregory B.\ Sorkin}
\address{Department of Mathematics, The London School of Economics and Political Science, Houghton Street, London WC2A 2AE, England}
\email{g.b.sorkin@lse.ac.uk}

\date{}

\begin{document}

\maketitle

\begin{abstract}
The semi-random graph process is an adaptive random graph process in which an online algorithm is initially presented an empty graph on $n$ vertices. In each round, a vertex $u$ is presented to the algorithm independently and uniformly at random. The algorithm then adaptively selects a vertex $v$, and adds the edge $uv$ to the graph. For a given graph property, the objective of the algorithm is to force the graph to satisfy this property asymptotically almost surely in as few rounds as possible.

We focus on the property of Hamiltonicity. We present an adaptive strategy which creates a Hamiltonian cycle in $\alpha n$ rounds, where $\alpha < 1.81696$ is derived from the solution to a system of differential equations. We also show that achieving Hamiltonicity requires at least $\beta n$ rounds, where $\beta > 1.26575$. 
\end{abstract}

\section{Introduction and Main Results}

The \textit{semi-random graph process} was suggested by Peleg Michaeli, introduced formally in 2020~\cite{ben2020semi}, and studied recently~\cite{ben-eliezer_fast_2020,gao2022perfect,behague2022,burova2022semi,macrury2022sharp,koerts2022k,molloy2024perfect,behague2024creating}, especially in the context of Hamiltonian cycles~\cite{gao2020hamilton,gao2022fully,frieze2022hamilton,molloy2024perfect}. It is an example of an \textit{adaptive} random graph process, in that an \textit{algorithm} has partial control over which random edges are added in each step. Specifically, the algorithm begins with the empty graph $G_0$ on vertex set $[n] = \{1, \ldots, n\}$, and in each \textit{step} (or round) $t \in \mb{N}$, a vertex $u_t$ is chosen independently and uniformly at random (u.a.r.) from $[n]$. The algorithm is \textit{online}, in that it is given the graph $G_{t-1}$ and vertex $u_t$ and must select a vertex $v_t \in [n] \setminus \{u_t\}$ and add the edge $(u_t,v_t)$ to $G_{t-1}$ to form $G_{t}$. Thus, it decides on $v_t$ without full knowledge of which vertices will be randomly drawn in the future. In this paper, the goal of the online algorithm is to build a multigraph satisfying a given graph property $\scr{P}$ as quickly as possible.  Clearly, if the online algorithm chooses $v_t$ u.a.r.\ for $m \ge 1$ consecutive rounds, then this is the Erdős–Rényi random graph process with multi-edges. The main focus in the literature is understanding how through intelligent decision-making, the online algorithm can speed up the appearance of certain graph properties $\scr{P}$.

While the semi-random graph process has been studied extensively in recent years, discrete processes in which an algorithm has partial control over its random steps have been studied previously. One of the first such examples is the work of Azar et al.~\cite{azar_balanced} in the context of the bin-packing problem. By allowing an algorithm a small amount of \textit{adaptivity}, Azar et al.~proved that the maximum load on any bin can be reduced by an exponential factor in comparison to a purely random (and non-adaptive) strategy. This phenomena has since been described as the ``power of two choices''. 

The \textit{Achlioptas process} is another example of an adaptive random process; it was proposed by Dimitris Achlioptas and first formally studied in \cite{bohman2001avoiding}. The Achlioptas process, too, begins with the empty graph $G_0$ on vertex set $[n]$. In each round $t \in \mb{N}$, the online algorithm is presented two distinct edges $e^{1}_t, e^{2}_t$ drawn u.a.r. from the edges on vertex set $[n]$ that were \textit{not} previously chosen (i.e., not in $G_{t-1}$).  The online algorithm then chooses precisely one of $e^{1}_t, e^{2}_t$, and then adds it to $G_{t-1}$, yielding the graph $G_{t}$. In contrast to the semi-random process, the objective first considered in~\cite{bohman2001avoiding} is to \textit{delay} the construction of a graph satisfying a property $\scr{P}$ for as many rounds as possible. Achlioptas asked what can be done if $\scr{P}$ corresponds to the existence of a giant component (i.e., a connected component of size $\Omega(n)$). Bohman and Frieze analyzed a greedy strategy which does \textit{not} build a giant component for $0.535 n$ rounds, which is strictly larger than the threshold at which a giant component appears in the Erdős–Rényi random graph process \cite{erdHos1960evolution}. Improvements have since been made on determining the optimal algorithm for delaying the appearance of a giant component. The best known lower bound of $0.829n$ is due to Spencer and Wormald~\cite{spencer_birth}, and the best known upper bound of $0.944n$ is due to Cobârzan~\cite{cob_2014}.  Numerous works have studied other properties and extensions of the Achlioptas process. Most related to our work is~\cite{krivelevich2010hamiltonicity}, whose goal is to force $G_t$ to be Hamiltonian as \textit{quickly} as possible. We refer the reader to the introduction of~\cite{Fraiman2022OnTP} for an in-depth overview of the literature on the Achlioptas process.

\subsection{Definitions} 

We now introduce some definitions and notation for the semi-random graph process. We formalize an online algorithm using a \textit{strategy} $\scr{S}$. A strategy $\scr{S}$ specifies for each $n \ge 1$,  a sequence of functions $(s_{t})_{t=1}^{\infty}$, where for each $t \in \mb{N}$, $s_t(u_1,v_1,\ldots, u_{t-1},v_{t-1},u_t)$ is a distribution on $[n] \setminus \{u_t\}$ which depends on the vertex $u_t$, and the \textit{history} of the process up until step $t-1$ (i.e., $u_1,v_1,\ldots, u_{t-1},v_{t-1}$). Then, $v_t$ is chosen according to this distribution, and $(u_t,v_t)$ is added to $G_{t-1}^{\scr{S}}(n)$, the multigraph constructed by $\scr{S}$ after the first $t-1$ steps. If $s_t$ is an atomic distribution, then $v_t$ is determined by $u_1,v_1, \ldots ,u_{t-1},v_{t-1},u_t$. We denote by $(G_{i}^{\scr{S}}(n))_{i=0}^{t}$ the sequence of random multigraphs obtained by following the strategy $\scr{S}$ for $t$ steps, and we shorten $G_{t}^{\scr{S}}(n)$ to $G_t$ or $G_{t}(n)$ when clear. 

Suppose $\scr{P}$ is an \textit{increasing} (i.e., \textit{monotone}) graph property. Given a strategy $\scr{S}$ and a constant $0<q<1$, let $m_{\scr{P}}(\scr{S},q,n)$ be the minimum $t \ge 0$ for which $\mb{P}[G_{t} \in \scr{P}] \ge q$, where $m_{\scr{P}}(\scr{S},q,n):= \infty$ if no such $t$ exists. Define
\[
m_{\scr{P}}(q,n) = \inf_{ \scr{S}} m_{\scr{P}}( \scr{S},q,n),
\]
where the infimum is over all strategies on $[n]$. As $\scr{P}$ is increasing, for each $n \ge 1$, if $0 \le q_{1} \le q_{2} \le 1$, then $m_{\scr{P}}(q_1,n) \le m_{\scr{P}}(q_2,n)$. Thus, the function $q \mapsto \limsup_{n\to\infty} m_{\scr{P}}(q,n)$ is non-decreasing, so the limit
\begin{align} \label{Cdef}
C_{\scr{P}}:=\lim_{q\to 1^-}\limsup_{n\to\infty} \frac{m_{\scr{P}}(q,n) }{n}
\end{align}
is guaranteed to exist. The goal is typically to compute upper and lower bounds on $C_{\scr{P}}$ for various properties $\scr{P}$.

\begin{remark}
Note that although $C_{{\mathcal P}}$ is well defined for all increasing properties, it only gives useful information if $C_{{\mathcal P}}$ is not equal to $0$ or $\infty$, i.e., if a linear number (in $n$) of steps is necessary and sufficient to construct some graph in ${\mathcal P}$. If $C_{{\mathcal P}}$ is equal to 0 or $\infty$ then the property is not linear and the definition \eqref{Cdef} should be adapted, scaling $m_{\scr{P}}(q,n)$ by an appropriate function of $n$ rather than $n$ itself in the denominator.
\end{remark}

\subsection{Main Results: Upper Bound} 

In this paper, we focus on the property of having a Hamiltonian cycle, which we denote by ${\tt HAM}$. It was shown in~\cite{bohman2009hamilton} that the the well-known $3$-out process generates a random graph that is Hamiltonian a.a.s.\ (asymptotically almost surely, i.e., with probability tending to 1 as $n \to \infty$). The very first paper on the semi-random process, ~\cite{ben2020semi}, showed that it can simulate the $k$-out process, from which it follows that $C_{\tt HAM} \le 3$. A new upper bound was obtained in~\cite{gao2020hamilton} in terms of an optimal solution to an optimization problem whose value is believed to be at most $2.61135$ by numerical support. 

The upper bound $C_{\tt HAM} \leq 3$ obtained by simulating the $3$-out process is \textit{non-adaptive}. That is, the strategy does {not} depend on the history of the semi-random process. The improvement in~\cite{gao2020hamilton} is adaptive but in a weak sense. The strategy consists of four phases, each lasting a linear number of rounds, and the strategy is adjusted {only} at the end of each phase: for example, the algorithm might identify vertices of low degree, and then focus on them during the next phase. 

In the proceedings version~\cite{gao2022fully} of this paper, a fully adaptive strategy was proposed: at every step $t$, it pays attention to $G_{t-1}$ and $u_t$. As expected, such a strategy creates a Hamiltonian cycle substantially faster, and it improves the upper bound from $2.61135$ to $2.01678$. A neat improvement in~\cite{frieze2022hamilton} brings the upper bound down to $1.84887$. In this paper, we combine all the ideas together to reduce it further, to $1.81701$.

\begin{theorem}\label{thm:main_upper_bound}
$C_{\texttt{HAM}} \le \alpha \le 1.81701$, where $\alpha$ is derived from a system of differential equations. 
\end{theorem}

The numerical results presented in this paper were obtained using the Julia programming language~\cite{Julia}. We would like to thank Bogumi\l{} Kami\'nski from SGH Warsaw School of Economics for helping us to implement it. The program is available on-line.\footnote{\texttt{https://math.torontomu.ca/\~{}pralat/}} 

\subsection{Main Results: Lower Bound}

Let us now turn to the lower bound. As observed in~\cite{ben2020semi}, if $G_t$ has a Hamiltonian cycle, then $G_t$ has minimum degree at least 2; thus $C_{\tt HAM} \ge  C_{\scr{P}}  = \ln 2+ \ln(1+\ln2) \ge 1.21973$, where $\scr{P}$ is the property of having minimum degree~$2$. In \cite{gao2020hamilton}, this was shown to not be tight: it was increased by a numerically negligible $10^{-8}$. By investigating some specific structures generated by the semi-random process, containing many edges that cannot simultaneously belong to a Hamiltonian cycle, we improve the lower bound of $\ln 2+ \ln(1+\ln2) \ge 1.21973$ to $1.26575$. (This bound was already reported in the proceedings version~\cite{gao2022fully} of this paper.) This is a much stronger bound than that in~\cite{gao2020hamilton}, the structures exploited are different, and the proof is simpler.

\begin{theorem}\label{thm:main_lower_bound}
Let $f(s)=2+e^{-3s}(s+1)\left(1-\frac{s^2}{2}-\frac{s^3}{3}-\frac{s^4}{8}\right)+e^{-2s}\left(2s+\frac{5s^2}{2}+\frac{s^3}{2}\right)-e^{-s}\left(3+2s\right)$, and let $\beta\approx 1.26575$ be the positive root of $f(s)-1=0$. Then, $C_{\texttt{HAM}} \ge \beta$.
\end{theorem}

\subsection{Further Related Works} 

The seminal paper~\cite{ben2020semi} showed that the semi-random graph process is general enough to simulate several well-studied random graph models by using appropriate strategies. In the same paper, the process was studied for various natural properties such as having minimum degree $k \in \Nn$ or having a fixed graph $H$ as a subgraph. In particular, it was shown that \emph{a.a.s.}\ one can construct $H$ in less than $n^{(d-1)/d} \omega$ rounds where $d \ge 2$ is the degeneracy of $G$ and $\omega = \omega(n)$ is any function that tends to infinity as $n \to \infty$. This property was recently revisited in~\cite{behague2022}, where a conjecture from~\cite{ben2020semi} was proven for any graph $H$: \emph{a.a.s.}\ it takes at least $n^{(d-1)/d} / \omega$ rounds to create $H$. The property of having cliques of order tending to infinity as $n\to \infty$ was investigated in~\cite{gamarnik2023cliques}. In \cite{koerts2022k}, $k$-factors and $k$-connectivity were studied. 

Another property studied in the context of semi-random processes is that of having a perfect matching, which we denote by {\tt PM}. Since the $2$-out process has a perfect matching \emph{a.a.s.}~\cite{WALKUP1980}, and the semi-random process can simulate the 2-out process, we immediately get that $C_{\texttt{PM}} \le 2$. By simulating the semi-random process with another random graph process known to have a perfect matching \emph{a.a.s.}~\cite{pittel}, the bound can be improved to $1+2/e < 1.73576$~\cite{ben2020semi}. This bound was recently improved by investigating another fully adaptive algorithm~\cite{gao2022perfect}, giving the current best bound of $C_{\texttt{PM}} < 1.20524$. The same paper improves the lower bound observed in~\cite{ben2020semi} of $C_{\texttt{PM}} \ge \ln(2) > 0.69314$ to  $C_{\texttt{PM}} > 0.93261$. While the optimal value of $C_{\texttt{PM}}$ remains unknown, a general purpose theorem is proven in~\cite{macrury2022sharp} that identifies a sufficient condition for a property $\scr{P}$ to have a sharp threshold. When $\scr{P}$ is  $\texttt{HAM}$ or $\texttt{PM}$, \cite{macrury2022sharp} uses the theorem to establish the existence of sharp thresholds for these properties.

Let us now discuss what is known about the property of containing a given spanning graph $H$ as a subgraph. It was asked by Noga Alon whether for any bounded-degree $H$, one can construct a copy of $H$ \emph{a.a.s.}\ in $O(n)$ rounds.  This question was answered positively in a strong sense in~\cite{ben-eliezer_fast_2020}, in which it was shown that any graph with maximum degree $\Delta$ can be constructed \emph{a.a.s.}\ in $(3\Delta/2+o(\Delta))n$ rounds and, if $\Delta = \omega (\log(n))$, in $(\Delta/2 + o(\Delta))n$ rounds. Note that these upper bounds are asymptotic in $\Delta$; when $\Delta$ is constant in $n$, such as for perfect matchings and Hamiltonian cycles, they give no concrete bound.

Other adaptive random graph processes and variants of the semi-random graph process have been considered in the literature. The \textit{semi-random tree process} is introduced in~\cite{burova2022semi}, where in each round, a random spanning tree of $K_n$ is presented to the algorithm, who chooses one of the edges to keep.  In~\cite{Harjas}, $k$ random vertices rather than just one are offered, and the algorithm chooses one of them before creating an edge. In~\cite{macrury2022sharp}, a general definition of an \textit{adaptive random graph process} is proposed. By parameterizing it appropriately, one recovers the Achlioptas process, the semi-random graph process, as well as the models of~\cite{burova2022semi} and~\cite{Harjas}. In~\cite{gilboa2021semi}, the vertices offered by the process follow a random permutation. Finally, hypergraphs are investigated in~\cite{behague2022,molloy2024perfect,behague2024creating}.

\section{Proof of Theorem \ref{thm:main_upper_bound}}

\subsection{Algorithmic Preliminaries} \label{sec:preliminaries}

In this section, we introduce some notation/terminology as well as the basic ideas used in the design of all of our strategies. We say that vertex $x \in [n]$ is \textit{covered} by $u_t$ arriving at round $t$, or that $u_t$ \textit{lands} on $x$, provided $u_t = x$. The main ingredient for proving Theorem~\ref{thm:main_upper_bound} is to specify a strategy which keeps ``extending'' or ``augmenting'' a path $P$ -- as will be explained momentarily -- as well as building a collection $\YY$ of edges, all vertex-disjoint from one another and~$P$, until all edges in $\YY$ are joined to $P$ and $P$ becomes Hamiltonian. Then, with a few more steps (just $o(n)$), the Hamiltonian path $P$ can be completed into a Hamiltonian cycle. 

Suppose that after $t \ge 0$ steps, we have constructed the graph $G_t$  which contains the path $P_t$ and the collection $\YY_t$ of disjoint edges. Let $V(P_t)$ and $V(\YY_t)$ denote respectively the vertices in $P_t$ and in $\YY_t$, and $U_t$ the vertices in neither. Denote the (induced) distance between vertices $x, y \in V(P_t)$ on the path $P_t$ by $d_{P_t}(x,y)$. We also define $d_{P_t}(x,Q) := \min_{q \in Q} d_{P_t}(x,q)$, for $x\in V(P_t)$ and $Q \subseteq V(P_t)$. In step $\tp$,   $u_{\tp} \in U_t, V(\YY_t)$, or $V(P_t)$:
\begin{itemize}
\item
If $u_\tp \in U_t$, we extend the collection $\YY_t$ by choosing $v_\tp$ to be a different vertex in $U_t$ and adding $u_\tp v_\tp$ to $\YY_t$. (If $\card{U_t}=1$, we simply ``pass'' on the round, choosing $v_\tp$ arbitrarily and not using $u_\tp v_\tp$ for the construction of the Hamiltonian cycle. However, this is unlikely to happen until the very end of the process, when we apply a different strategy; see Section~\ref{sec:clean_up}.) We call such a move a \textit{(greedy) $\YY$-extension}.
\item 
If $u_\tp \in V(\YY_t)$, an extension of $P_t$ can be made by appending the $\YY_t$ edge on $u_\tp$ to an end of $P_t$, and  deleting it from $\YY_t$. Such a move is called a \textit{(greedy) path extension}.  
\item
If $u_\tp \in V(P_t)$ we cannot perform a greedy path extension, but we can still choose $v_\tp$ in a way that will help us extend the path in future rounds. Specifically, choose $v_\tp \in U_{t} \cup V(\YY_t)$ to be an isolated vertex or an edge endpoint; it could be chosen uniformly at random, though we will use a more efficient strategy. Consider a future round $i>t$ where $u_\ip$ happens to be a path neighbour of $u_\tp$ (i.e., $d_{P_i}(u_\tp,u_\ip) =1$). In this case, if $v_\tp \in U_i$ is an isolated vertex, set $v_\ip = v_\tp$, and replace the path edge $\set{u_\tp,u_{i+1}}$ with the length-2 path $(u_\tp,v_\tp= v_\ip , u_\ip)$, thus making $P$ one edge longer. If $v_\tp \in V(\YY_i)$ belongs to an isolated edge, choose $v_\ip$ to be its neighbour, and replace the path edge $\set{u_\tp, u_\ip}$ with the length-3 path $(u_\tp, v_\tp, v_{i+1}, u_{i+1})$, making $P$ two edges longer. Call either of these cases a \textit{path augmentation}.
\end{itemize}

\subsection{Proof Overview} \label{sec:upper_bound_overview}

In order to prove Theorem~\ref{thm:main_upper_bound}, we analyze a strategy which proceeds in three distinct \textit{stages}. In the first stage, we execute \ref{alg:degree_greedy}, an algorithm which makes greedy $\YY$-extensions and path extensions whenever possible, and otherwise sets up path augmentation operations for future rounds in a degree-greedy manner. During the execution of\ \ \ref{alg:degree_greedy} some edges are coloured red or blue to help keep track of when these augmentations can be made. We use two colours, namely red and blue, to distinguish between edges which are added randomly (red) and greedily (blue). In step $\tp$, $v_\tp$ is chosen amongst $U_t\cup V(\YY_t)$ that are incident with the least number of blue edges. This degree-greedy decision is done to minimize the number of coloured vertices which are destroyed when path augmentations and extensions are made in later rounds. This stage lasts for $N$ \textit{phases}, where $N$ is any non-negative integer that may be viewed as the parameter of the algorithm (here a phase is a contiguous set of steps shorter than the full stage). For the claimed (numerical) upper bound of Theorem~\ref{thm:main_upper_bound}, $N$ is set to $100$. Setting smaller values of the parameter $N$---in particular, setting $N=0$---yields an algorithm that is easier to analyse. Setting $N > 100$ can slightly improve the bound in Theorem~\ref{thm:main_upper_bound}, but the gain is rather insignificant. 

The output of the first phase are $P$ and $\mathcal{Y}$ that have been constructed,  together with the set $\mathcal{E}$ of red edges (all blue edges will be discarded). The second stage takes $(P,\mathcal{Y},\mathcal{E})$ as input, and executes a procedure called \ref{alg:fully_randomized}, an algorithm which makes greedy $\YY$-extensions or path extensions whenever possible, and otherwise chooses $v_\tp$ randomly amongst $U_t\cup V(\YY_t)$. We execute \ref{alg:fully_randomized} until we are left with $\eps n$ vertices in $U_t\cup V(\YY_t)$, where $\eps=\eps(n)$ tends to $0$ as $n \rightarrow \infty$ arbitrarily slowly. (In practice, one can set $\eps$ to be an arbitrarily small positive number when running this algorithm.) At this point, we proceed to the final stage where a clean-up algorithm is run, which uses merely path augmentations. Using well-known concentration inequalities we prove that a Hamiltonian cycle can be  constructed in an additional $O(\sqrt{\eps}n)=o(n)$ steps.

In Section~\ref{warm-up-upper-bound}, we first describe \ref{alg:fully_randomized}, as it is easier to state and analyze than \ref{alg:degree_greedy}. Moreover, if we take $N=0$, which corresponds to executing \ref{alg:fully_randomized} from the beginning, then we will be left with a path on all but $\eps n$ vertices after $\alpha^{*}n$ steps where $\alpha^{*} \le 1.84887$. This is exactly the upper bound obtained in~\cite{frieze2022hamilton}. Our third stage clean-up algorithm from Section \ref{sec:clean_up} allows us to complete the Hamiltonian cycle in another $o(n)$ steps. Thus, Sections~\ref{warm-up-upper-bound} and~\ref{sec:clean_up} provide a self-contained proof of an upper bound on $C_{\mathtt{HAM}}$ of $\alpha^{*} \le 1.84887$ (see Theorem~\ref{thm:warm_up}). Afterwards, in Section~\ref{sec:degree_greedy} we formally state and analyze our first stage algorithm. This is the most technical section of the paper, as \ref{alg:degree_greedy} makes decisions in a more intelligent manner than \ref{alg:fully_randomized} which necessitates more random variables in its analysis. By executing these three stages in the aforementioned order, we attain the claimed upper bound of Theorem~\ref{thm:main_upper_bound}.

\subsection{A Fully Randomized Algorithm}\label{warm-up-upper-bound} 

The algorithm takes a tuple $(P,\mathcal{Y},\mathcal{E})$ as input where 
\begin{itemize}
    \item $P$ is a path on a subset of vertices in $[n]$;
    \item $\mathcal{Y}$ is a set of pairs of vertices in $[n]\setminus V(P)$;
    \item $\mathcal{E}$ is a set of \textit{red} edges; each red edge has exactly one endpoint on $P$ and no two edges in $\mathcal{E}$ are adjacent to the same vertex in $P$; moreover, these endpoints are at distance at least 3 on the path from each other.
\end{itemize}
In order to simplify the analysis, we begin the semi-random graph process from round $t = 0$ with the initial graph $G_0$ induced by $(P, \mathcal{Y},\scr{E})$. Note that if $N=0$, then $G_0$ is the empty graph on $[n]$. We encourage the reader to keep this case in mind
on a first read through.

When considering $G_t$, a certain subset of its edges will be coloured red. This helps us define certain vertices used by our algorithm for path augmentations. A vertex $x \in V(P_t)$ is \textit{one-red}  provided it is adjacent to precisely one red edge of $G_t$. Similarly, $x \in V(P_t)$ is \textit{two-red}, provided it is adjacent to precisely two red edges of $G_t$. Throughout the execution of \ref{alg:fully_randomized}, each vertex in $V(P_t)$ is incident with at most two red edges. We denote the sets of one-red vertices and two-red vertices by $\scr{L}^{1}_t$ and $\scr{L}^{2}_t$, respectively, and refer to $\scr{L}_t := \scr{L}^{1}_t \cup \scr{L}^{2}_t$ as the \textit{red} vertices of $G_t$. By definition, $\scr{L}^{1}_t$ and $\scr{L}^{2}_t$ are disjoint. Initially, $P_0=P$, $\mathcal{Y}_0=\mathcal{Y}$, $\scr{L}^{2}_0=\emptyset$, and $\scr{L}^{1}_0$ is set to be the set of vertices in $V(P)$ that are incident with a red edge in ${\mathcal E}$. It will also be convenient to maintain a set of \textit{permissible vertices} $\scr{Q}_t \subseteq V(P_t)$ which specifies which uncoloured vertices on the path can be turned red. In order to simplify our analysis, we specify the size of $\scr{Q}_t$ and ensure that it only contains vertices of path distance at least $3$ from the red vertices on $P_t$. Formally:
\begin{enumerate}
    \item[(i)] $|\scr{Q}_t| = \max\{ |V(P_t)| - 5|\scr{L}_t| ,0\}$. \label{eqn:permissible_size}
    \item[(ii)] If $\scr{L}_t \neq \emptyset$, then each $x \in \scr{Q}_t$ satisfies $d_{P_t}(x,\scr{L}_t) \ge 3$. \label{eqn:permissible_distance}
\end{enumerate}
When $\scr{L}_t = \emptyset$, we simply take $\scr{Q}_t = V(P_t)$. Otherwise, since $|\{ x \in V(P_t) : d_{P_t}(x,\scr{L}_t) \le 2 \}| \le 5 |\scr{L}_t|$, we can maintain these properties by initially taking $\{ x \in V(P_t) : d_{P_t}(x,\scr{L}_t) \ge 3\}$, and then (if needed) arbitrarily removing $|\{ x \in V(P_t) : d_{P_t}(x,\scr{L}_t) \ge 3\}| - \max\{ |V(P_t)| - 5|\scr{L}_t| ,0\}$ vertices from it.

Upon the arrival of $u_\tp$, there are five main cases our algorithm must handle. The first three cases involve extending $P_t$ or $\YY_t$, whereas the latter two describe what to do when it is not possible to extend the path in the current round, and how the one-red and two-red vertices are created.

\begin{enumerate}
    \item If $u_\tp$ lands within $U_t$, then choose $v_\tp$ u.a.r., and greedily extend $\mathcal{Y}_t$ unless $|U_t|=1$. \label{eqn:greedy_extend_deter}
    \item If $u_\tp$ lands in $V(\YY_t)$, then greedily extend $P_t$.
    \item If $u_\tp$ lands at path distance one from some $x \in \scr{L}_t$, then augment $P_t$ via an arbitrary red edge of $x$. \label{eqn:path_augment_deter}
    \item If $u_\tp$ lands in $\scr{Q}_t$, then choose $v_\tp$ u.a.r.\ amongst $U_t \cup V(\YY_t)$, and colour $u_\tp v_\tp$ red. This case creates a one-red vertex. \label{eqn:one_cheery} 
    \item If $u_\tp$ lands in $\scr{L}^{1}_t$, then choose $v_\tp$ u.a.r.\ amongst $U_t$ and colour $u_\tp v_\tp$ red. This case converts a one-red vertex to a two-red vertex. 
\end{enumerate}
In all the remaining cases, we choose $v_\tp$ arbitrarily, and interpret the algorithm as \textit{passing} on the round, meaning the edge $u_\tp v_\tp$  will not be used to construct a Hamiltonian cycle. In particular, the algorithm passes on rounds in which $u_\tp$ lands at path distance two from some $x \in \scr{L}_t$. This guarantees that no two red vertices are at distance two from each other and so when $u_\tp$ lands next to a red vertex, this neighbouring red vertex is uniquely identified. Let us say that a red vertex is \textit{well-spaced}, provided it is at distance at least $3$ on the path from all other red vertices, and it is \textit{not} an endpoint of $P_t$. Observe that each well-spaced red vertex yields precisely two vertices on $P_t$ where a path augmentation involving $u_\tp$ can occur. By construction, all but at most $2$ of the algorithm's red vertices are well-spaced. The step $\tp$ of the algorithm when $u_\tp$ is drawn u.a.r.\ from $[n]$ is formally described by the \ref{alg:fully_randomized} algorithm. Specifically, we describe how the algorithm chooses $v_\tp$, how it constructs $P_\tp$, and how it adjusts the colours of $G_\tp$, thus updating $\scr{L}^{1}_t$ and $\scr{L}^{2}_t$. 
\begin{varalgorithm}{$\mathtt{FullyRandomized}$}
\caption{Step $\tp$} 
\label{alg:fully_randomized}
\begin{algorithmic}[1]
\If{$u_\tp \in U_t$ and $|U_t| \ge 2$} \Comment{greedily extend $\YY_t$}
\State Let $v_\tp$ be a uniformly random vertex in $U_t\setminus \{u_\tp\}$.
\State Set $P_\tp = P_t$ and $\YY_\tp=\YY_t\cup \{u_\tp v_\tp\}$.  

\ElsIf{$u_\tp y\in \YY_t$ for some $y$}\Comment{greedily extend $P_t$}
\State Let $v_\tp$ be an arbitrarily chosen endpoint of $P_t$.

\State Set $V(P_\tp) = V(P_t) \cup \{u_\tp,y\}$, $E(P_\tp) = E(P_t) \cup \{u_\tp v_\tp,u_\tp y \}$. 
\State Set $\YY_\tp=\YY_{t}\setminus \{u_\tp y\}$.
\State Uncolour all of the edges adjacent to $u_\tp$ or $y$.

\ElsIf{$d_{P_t}(u_\tp, \scr{L}_t) =1$}     \Comment{path augment via red vertices}
\State Let $x \in \scr{L}_t$ be the (unique) red vertex adjacent to $u_\tp$ 

\State Denote $x y \in E(G_t)$ an arbitrary red edge of $x$.
\If{$y \in U_t$}
\State Set $v_\tp = y$.
\State Set $V(P_\tp) = V(P_t) \cup \{ v_\tp \}$ and $E(P_\tp) = (E(P_t) \cup \{x v_\tp, u_\tp v_\tp\}) \setminus \{u_\tp x \}$. \State Set $\YY_\tp=\YY_{t}$
\State  Uncolour all of the edges adjacent to $r$.
\ElsIf{$yy' \in \scr{Y}_t$ for some $y'$}
\State Set $v_\tp=y'$
\State Set $V(P_\tp) = V(P_t) \cup \{ y, v_\tp \}$, $E(P_\tp) = (E(P_t) \cup \{xy, y v_\tp, u_\tp v_\tp \}) \setminus \{u_\tp x \}$. 
\State Set $\YY_\tp=\YY_t\setminus\{y v_\tp\}$
\State Uncolour all of the edges adjacent to $y$ or $v_\tp$.

\EndIf

\Else{}
\If{$u_\tp \in \scr{Q}_t \cup \scr{L}^1_t$}  \Comment{construct red vertices or pass}
\State Choose $v_\tp$ u.a.r.\ from $U_t \cup V( \YY_t )$. 
\State Colour $u_\tp v_\tp$ red.    \Comment{construct a one-red or two-red vertex}  

\Else{} \Comment{pass on $u_\tp v_\tp$}
\State Choose $v_\tp$ arbitrarily from $[n]$. 
\EndIf

\State Set $P_\tp = P_t$; $\YY_\tp=\YY_{t}$.

\EndIf

\State Update $U_{t+1}$ and $\scr{L}_{t+1}$.

\State Update $\scr{Q}_\tp$, if needed, such that $|\scr{Q}_\tp| = |V(P_\tp)| - 5|\scr{L}_\tp|$.

\end{algorithmic}
\end{varalgorithm}

We define the random variables $X(t)= |V(P_t)|$, $L_{1}(t)= |\scr{L}^{1}_t|$, $L_{2}(t) = |\scr{L}^{2}_t|$, $L(t)=|\scr{L}_t|=L_{1}(t)+L_{2}(t)$, and $Y(t)=|V(\YY_t)|=2|\YY_t|$. Note that $L(t)$ is an auxiliary random variable which we define only for convenience, and $Y(t)$ denotes the number of vertices incident to edges in $\YY_t$. 

The input $(P,\scr{Y},{\mathcal E})$ of \ref{alg:fully_randomized} is the output of \ref{alg:degree_greedy}, and thus is randomized. Our analysis of the execution of \ref{alg:fully_randomized} relies on the fact that $(P,\YY, {\mathcal E})$ has a certain distribution. To be specific, recall that ${\mathcal L}^{1}_0$ is the set of vertices on $P$ incident with an edge in ${\mathcal E}$. Then, conditional on $P$, $|\YY|$, and ${\mathcal L}^{1}_0$, the following properties are satisfied by $(P,\YY, {\mathcal E})$:
\begin{enumerate} [label=(\subscript{O}{{\arabic*}})]
    \item $\YY$ is uniform over all possible $|\YY|$ pairs of vertices in $[n]\setminus V(P)$; \label{item:pair_distribution}
    \item ${\mathcal E}$ is uniform over all possible set of edges joining ${\mathcal L}^{1}_0$ and $[n]\setminus V(P)$ such that every vertex in ${\mathcal L}^{1}_0$ is incident with exactly one edge in ${\mathcal E}$. \label{item:red_edge_distribution}
\end{enumerate}
We shall prove that these properties hold in \Cref{sec:proof_main}. Using these properties, together with the specification of \ref{alg:fully_randomized}, we first show that our random variables cannot change drastically in one round.  We use $\Delta$ to denote the one step changes in our random variables (i.e., $\Delta X(t): = X(\tp) - X(t)$).

\begin{lemma}[Boundedness Hypothesis -- \ref{alg:fully_randomized}] \label{lem:lipschitz_randomized}
With probability $1 - O(n^{-1})$, 
$$\max\{ |\Delta X(t)|, |\Delta L_{1}(t)|, |\Delta L_{2}(t)|, |\Delta Y(t)|\} = O(\log n) $$
for all $0 \le t \le 3n -|\mathcal{E}|$ with $n -X(t) \ge n/ \log n$. 
\end{lemma}

\begin{proof} 
Note that, by design, the path can only increase its length but it cannot absorb more than two vertices in each round. Hence, the desired property clearly holds for the random variable $X(t)$. The same holds for $Y(t)$. To estimate the maximum change for the random variables $L_1(t)$ and $L_2(t)$, we need to upper bound the number of red edges adjacent to any particular vertex $v\in U_t\cup V(\YY_t)$. Observe that due to \ref{item:red_edge_distribution} and how the red edges are randomly selected, if we condition on $U_t \cup V(\YY_t)$, then this is stochastically upper bounded by $\Bin(t +|\mathcal{E}|, |U_t \cup V(\YY_t)|^{-1})$. Since $t +|\mathcal{E}| \le 3n$, and we have assumed that there are at least $n/ \log n$  vertices in $U_t\cup V(\YY_t)$, the number of red edges adjacent to $v$ is stochastically upper bounded by the binomial random variable $\Bin(3n, \log n/n $) with expectation $3 \log n $. It follows immediately from Chernoff's bound that with probability $1-O(n^{-3})$, the number of red edges adjacent to $v$ is $O(\log n)$, and so the desired bound holds by union bounding over all $3n^2$ vertices and steps.
\end{proof}

Let us denote $H_t = ( X(i), L_{1}(i), L_{2}(i), Y(i))_{0 \le i \le t}$. Note that $H_t$ does \textit{not} encompass the entire history of the random process after $t$ rounds (i.e., $G_0, \ldots ,G_t$, the first $\tp$ graphs appearing in the sequence generated by the process). The distribution of $(P,\YY,{\mathcal E})$ together with the technique of deferred information exposure permit a tractable analysis of the random positioning of $v_t$ when $u_t$ is red. In particular, as we only expose $Y(t)$ instead of $\YY_t$, $\YY_t$ has the same distribution (conditional on $H_t$) as first exposing the set of vertices in $[n]\setminus V(P_t)$, then uniformly selecting a subset of vertices in $[n]\setminus V(P_t)$ of cardinality $Y(t)$, and then finally taking a uniformly random perfect matching over the $Y(t)$ vertices (i.e.\ pair the $Y(t)$ vertices into $Y(t)/2$ disjoint edges). Similarly, conditional on $L_1(t)$ and $L_2(t)$, we may first expose ${\mathcal L}_t^1$ and ${\mathcal L}_t^2$, and then choose their neighbours joined by a red edge uniformly from $[n]\setminus V(P_t)$. Moreover, this process (of choosing the ends of red edges lying in $[n]\setminus V(P_t)$) is independent of the process of choosing  and pairing vertices for $\YY_t$. We observe the following expected difference equations. 

\begin{lemma}[Trend Hypothesis -- \ref{alg:fully_randomized}] \label{lem:randomized_expected_differences}
For each $t \ge 0$, if $n -X(t) \ge n/ \log n$, then by setting $\Gamma(t)=1+Y(t)/(n-X(t))$ and $A(t)=2Y(t)/(n-X(t))$,
\begin{eqnarray}
\E[\Delta X(t) \mid H_t]&=& \frac{2 Y(t)}{n} + \frac{2L(t)}{n}\cdot \Gamma(t)+O(\log n/n)\label{diff-X}\\ 
\E[\Delta Y(t) \mid H_t]&=& -\frac{2 Y(t)}{n} + \frac{2(n-X(t)-Y(t))}{n} - \frac{2L(t)}{n}\cdot A(t) \label{diff-Y} +O(\log n/n)\\ 
\E[ \Delta L_{1}(t)\mid H_t]&=&  \frac{X(t) - 5L(t)}{n} - \frac{2 L_{1}(t)}{n} + \frac{2 L_{1}(t)}{n}\left( \frac{2 L_{2}(t)}{n - X(t)} - \frac{L_{1}(t)}{n-X(t)}  \right) \cdot \Gamma(t) \nonumber\\
&& +\frac{2 L_{2}(t)}{n} + \frac{2 L_{2}(t)}{n}\left(  \frac{2 L_{2}(t)}{n - X(t)} - \frac{L_{1}(t)}{n-X(t)} \right) \cdot \Gamma(t) - \frac{L_{1}(t)}{n} \nonumber\\
&& + \frac{2Y(t)}{n} \cdot\left( \frac{2 L_{2}(t)}{n-X(t) } - \frac{L_{1}(t)}{n-X(t) } \right)   + O(\log n/n)\label{diff-L1}\\
\E[ \Delta L_{2}(t) \mid H_t] &=& \frac{L_{1}(t)}{n} -  \frac{2 L_{2}(t)}{n}\cdot A(t)  - \frac{2 L_{1}(t)}{n} \cdot \frac{2 L_{2}(t)}{n-X(t)} \cdot B(t)\nonumber\\
&& - \frac{2 L_{2}(t)}{n} - \frac{2 L_{2}(t)}{n} \cdot  \frac{2 L_{2}(t)}{n - X(t)} \cdot \Gamma(t) + O(\log n/n).\label{diff-L2}
\end{eqnarray}
\end{lemma}

\begin{proof}
As discussed earlier, the \ref{alg:fully_randomized} algorithm ensures that at time $t$ there are at most $2$ red vertices which are not well-spaced. Thus, since our expected differences each allow for a $O(\log n/n)$ term, without loss of generality, we can assume that all our red vertices are well-spaced.  Note also that all our explanations below assume that we have conditioned on $H_t$. 

When path augmentation occurs via a red edge incident to a vertex $x$ on $P_t$, we first expose $r$ (the other end of the red edge) which is distributed uniformly over all vertices in $[n]\setminus V(P_t)$, and then we expose whether $r$ is in $V(\YY_t)$. In the case that $r$ is in $V(\YY_t)$ we expose $y$ which is paired to $r$ in $\YY_t$.

The first expected difference is easy to see. Observe that there are three disjoint cases where $\Delta X(t)$ is nonzero. Case 1: $u_\tp$ lands on a vertex in $V(\YY_t)$. In this case $\Delta X(t)$ is 2, and this event occurs with probability $Y(t)/n$. Case 2: $u_\tp$ is next to a red vertex $x$ (i.e.\ a vertex in ${\mathcal L}_t$) on $P_t$, and path augmentation is performed via a red edge $xr$ where $r\in U_t$. In this case, $\Delta X(t)=1$ and the probability of this event is $(2L(t)/n)\cdot (1-Y(t)/(n-X(t)))$, where $2L(t)/n$ is the probability that $d_{P_t}(u_\tp,{\mathcal L}_t)=1$, and $1-Y(t)/(n-X(t))$ is the probability that $r\in U_t$ conditional on $r\in [n]\setminus V(P_t)$. Case 3: same as case 2 but $r\in V(\YY_t)$. In this case, $\Delta X(t)=2$ and the probability of this event is $(2L(t)/n)\cdot (Y(t)/(n-X(t)))$. Combining all three cases together we obtain~(\ref{diff-X}). 

The remaining equations are obtained in a similar manner. In what follows, we explain the event $A$ for which each term in the equations accounts for as $\E[\Delta Z(t) \cdot \bm{1}_{A}]$ where $Z\in\{Y,L_1,L_2\}$.

In the second equation, $-2Y(t)/n$ is the contribution from the case where $u_\tp$ lands on $V(\YY_t)$; $(-2L(t)/n)\cdot A(t)$ corresponds to the event that $u_\tp$ lands on a neighbour of some $x\in{\mathcal L}_t$ on $P_t$, and the path augmentation is performed via a red edge $xr$ where $r\in V(\YY_t)$. Finally, $2(n-X(t)-Y(t))/n$ corresponds to the event that $u_\tp$ lands on a vertex in $U_t$.

In the third equation, the term $(X(t)-5L(t))/n$ is the contribution from the case where $u_\tp$ lands on ${\mathcal Q}_t$. In the case where $u_\tp$ lands on a vertex neighbouring some $x\in {\mathcal L}_t^1$ on $P_t$ (which occurs with probability $2L_1(t)/n$) the contribution to $\Delta L_1(t)$ can come from two sources: (a) $x$ is removed from ${\mathcal L}_t^1$ after the path augmentation and thus it contributes $-1$ to  $\Delta L_1(t)$; (b) one or two vertices will be added to $P_t$, which results in uncolouring of all red edges incident to them, and which consequently contributes to $\Delta L_1(t)$. Note that $2L_2(t)/(n-X(t))$ is the expected number of two-red vertices that become one-red when a vertex $r\in U_t\cup V(\YY_t)$ is moved to $P_t$. Similarly, $L_1(t)/(n-X(t))$ is the expected number of one-red vertices that get removed from ${\mathcal L}_t^1$ due to moving a certain vertex $r\in U_t\cup V(\YY_t)$ to $P_t$. Finally, $B(t)$ is the expected number of vertices in $U_t\cup V(\YY_t)$ that will be added to $P_t$. 

The next two terms (in the second line of third equation) correspond to the symmetric case: $u_\tp$ lands on a vertex neighbouring some $x\in {\mathcal L}_t^2$ on $P_t$ (which occurs with probability $2L_2(t)/n$) with the same two sources that contribute to $\Delta L_1(t)$. The term $-L_1(t)/n$ corresponds to the case where $u_\tp$ lands on a vertex $x\in {\mathcal L}_t^1$ which results in moving $x$ from ${\mathcal L}_t^1$ to ${\mathcal L}_\tp^2$. 

Finally, the last term $(2Y(t)/n)(2L_2(t)-L_t(t))/(n-X(t))$ is the contribution from moving two vertices from $U_t\cup V(\YY_t)$ to $P_t$ in the case where $u_\tp$ lands on a vertex in $V(\YY_t)$.

For the last equation, $L_1(t)/n$ accounts for $u_\tp$ landing on a vertex in ${\mathcal L}^1_{t}$, and a one-red vertex becomes two-red. The term $-2L_2(t) A(t)/n$ accounts for the case where $u_\tp$ lands on a vertex in $V(\YY_{t})$.  In this case (which occurs with probability $Y(t)/2$) 2 vertices are moved from $V(\YY_{t})$ to $P_t$, each of which will be resulting in uncolouring $2L_2(t)/(n-X(t))$ red edges incident to vertices in ${\mathcal L}^2_t$ in expectation. 

The third and the fifth terms together in the equation account for the case where $u_\tp$ lands on a neighbour of ${\mathcal L}_t$, where one or two vertices in $U_t\cup V(\YY_t)$ are moved to $P_t$ after the path augmentation, each resulting in uncolouring $2L_2(t)/(n-X(t))$ red edges incident to ${\mathcal L}_t^2$ in expectation. The fourth term accounts for the case where $u_\tp$ lands on a vertex neighbouring a vertex $x\in {\mathcal L}_t^2$, resulting in the removal of $x$ from ${\mathcal L}_t^2$ after the path augmentation. 
\end{proof}

In order to analyze \ref{alg:fully_randomized}, we shall employ the differential equation method~\cite{de}. This method is commonly used in probabilistic combinatorics to analyze random processes that evolve step by step. The step changes must be small in relation to the entirety of the discrete structure. For instance, in our application, this refers to adding one edge at a time to the graph on $n$ vertices. The method allows us to derive tight bounds on the associated random variables which hold a.a.s.\ at every step of the random process. We refer the reader to~\cite{bennett2022gentle} for a gentle introduction to the methodology. 

Recall that \ref{alg:fully_randomized} takes input $(P,\YY, {\mathcal E})$. Let $X(0), Y(0), L_1(0)$ denote the number of vertices on $P$, the number of vertices incident to edges in $\YY$, and the number of vertices incident with ${\mathcal E}$, respectively.  We prove in Section~\ref{sec:proof_main} that there exist some constants $\hat{x}, \hat{y}, \hat{\ell_1}$ such that $|X(0)/n - \hat{x}|, |Y(0)/n - \hat{y}|, |L_1(0)/n - \hat{\ell_1}| \le \lambda$ for some $\lambda = o(1)$. Initially, there are no two-red vertices, that is, we will always set $L_2(0) = 0$. Let us now fix a sufficiently small constant $\eps>0$, and define the bounded domain 
$$
\scr{D}_{\eps}:= \{ (s,x,y,\ell_1, \ell_2):  -1 < s < 3, -1 < x < 1 - \eps, |\ell_1| < 2, |\ell_2| <2\}.
$$
Consider the system of differential equations in variable $s$ with functions $x =x(s), y=y(s), \ell_1 = \ell_{1}(s)$,  and $\ell_{2}=\ell_{2}(s)$:
\begin{eqnarray} 
    x' &=& 2y + 2 (\ell_{1} + \ell_2)\lambda \label{ode1} \\
    y' &=& -2y + 2(1-x-y) - 2 (\ell_{1} + \ell_2)a \label{odey} \\
    \ell_{1}' &=& x - 5 (\ell_1 + \ell_2) -2\ell_1 + (2\ell_1 \lambda+2\ell_2 \lambda+2y) \cdot \frac{2 \ell_2 - \ell_1 }{1-x}  
     +2\ell_2 -\ell_1   \label{ode2}  \\
    \ell_2' &=& \ell_1 - 2 \ell_2 a - (2\ell_1+2\ell_2) \lambda \cdot \frac{2 \ell_2}{1-x}  -2 \ell_2, \label{ode3} 
\end{eqnarray}
where $\lambda(s)=1+y(s)/(1-x(s))$ and $a(s)=2y(s)/(1-x(s))$. The right-hand side  of each of the above equations is Lipchitz on the domain $\scr{D}_{\eps}$. Define 
\[
T_{\scr{D}_{\eps}}=\min\{t\ge 0: \ (t/n, X(t)/n, Y(t)/n, L_1(t)/n,L_2(t)/n)\notin \scr{D}_{\eps}\}.
\]
Now, the `Initial Condition' of Theorem \ref{thm:differential_equation_method} is satisfied with values $(0,\hat{x}, \hat{y}, \hat{\ell}_1,0)$ and some $\lambda = o(1)$. Moreover, the `Trend Hypothesis' and `Boundedness Hypothesis' are satisfied with some $\delta=O(\log n/n)$, $\beta=O(\log n)$ (with failure probability $\gamma=o(1)$ throughout the process) by Lemmas~\ref{lem:lipschitz_randomized} and~\ref{lem:randomized_expected_differences}. Thus, for every $\delta>0$, $X(t)=nx(t/n)+o(n)$, $Y(t)=ny(t/n)+o(n)$, $L_1(t)=n\ell_1(t/n)+o(n)$ and $L_2(t)=n\ell_2(t/n)+o(n)$ uniformly for all $t_0 \le t \le(\sigma(\eps)-\delta) n$, where $x$, $y$, $\ell_1$ and $\ell_2$ are the unique solution to~(\ref{ode1})--(\ref{ode3}) with initial conditions $x(0)=\hat{x}$, $y(0)=\hat{y}$,  $\ell_1(0)=\hat{\ell}_1$, and $\ell_2(0)=0$, and $\sigma(\eps)$ is the supremum of $s$ to which the solution can be extended before reaching the boundary of $\scr{D}_{\eps}$. 

\begin{lemma}[Concentration of \ref{alg:fully_randomized}'s Random Variables] \label{lem1:concentration_random}
For every $\delta>0$, a.a.s.\ for all $0 \le t\le (\sigma(\eps)-\delta)n$,
\[
    \max \Big\{ |X(t) -x(t/n) n|, |Y(t) -y(t/n) n|,|L_{1}(t) - \ell_{1}(t/n) n|, |L_{2}(t) - \ell_{2}(t/n) n| \Big\} =  o(n). 
\]
\end{lemma}

As $\scr{D}_{\eps}\subseteq \scr{D}_{\eps'}$ for every $\eps>\eps'$, $\sigma(\eps)$ is monotonically nondecreasing as $\eps\to 0$. Thus, 
\begin{equation}
    \alpha^*:=\lim_{\eps \to 0+} \sigma(\eps) \label{alpha_star}
\end{equation} 
exists. It is clear that $|L_1(t)/n|$, $|L_2(t)/n|$, and $|Y(t)/n|$ are all bounded by 1 for all $t$ and thus, when $t/n$ approaches $\alpha^*$, either $X(t)/n$ approaches 1 or $t/n$ approaches 3. Formally, we have the following proposition.

\begin{proposition}\label{p:boundary} 
For every $\eps>0$, there exists $\delta>0$ such that a.a.s.\ one of the following holds.
\begin{itemize}
    \item $X(t)>(1-\eps)n$ for all $ t\ge (\alpha^*-\delta)n$;
    \item $\alpha^*=3$.
\end{itemize}
\end{proposition}

The ordinary differential equations~(\ref{ode1})--(\ref{ode3}) do not have an analytical solution. In both cases, $N=0$ and $N=100$, numerical solutions show that $\alpha^*< 1.85$. (For $N=0$, $\alpha^* \approx 1.84887$.) Thus, by the end of the execution of \ref{alg:fully_randomized}, there are $\eps n$ unsaturated vertices (i.e.\ vertices not in $P_t$) remaining, for some $\eps=o(1)$.

\subsection{A Clean-up Algorithm} \label{sec:clean_up}

Suppose that we are presented a path $P$ on $(1-\eps)n$ vertices of $[n]$, where $0 < \eps =\eps(n) < 1/1000$. The assumption on $\eps$ is a mild but convenient one. We will apply the argument for $\eps=o(1)$. In this section, we provide an algorithm for the semi-random graph process which absorbs the remaining $\eps n$ vertices into $P$ to form a Hamiltonian path, after which a Hamiltonian cycle can be constructed. The whole procedure takes $O(\sqrt{\eps}n + n^{3/4}\log^2 n) = o(n)$ further steps in the semi-random graph process. Moreover, the algorithm is self-contained in that it only uses the edges of $P$ in its execution.

\begin{lemma}[Clean-up Algorithm] \label{lem:clean_up}
Let $0 < \eps = \eps(n) < 1/1000$, and suppose that $P$ is a path on $(1-\eps)n$ vertices of $[n]$. Then, given $P$ initially, there exists a strategy for the semi-random graph process which builds a Hamiltonian cycle from $P$ in $O(\sqrt{\eps}n + n^{3/4}\log^2 n)$ steps a.a.s.
\end{lemma}

\begin{remark}
The constant hidden in the $O(\cdot)$ notation does not depend on $\eps$. The strategy used in the clean-up algorithm is similar to that in~\ref{alg:fully_randomized} but the analysis is done in a much less accurate way, as we only need to prove an $o(n)$ bound on the number of steps required to absorb $\eps n$ vertices, assuming $\eps = \eps(n) \to 0$ as $n \to \infty$.
\end{remark}

\begin{proof} [Proof of Lemma~\ref{lem:clean_up}]
Let $j_0=\eps n$. For each $k\ge 1$, let $j_k=(1/2) j_{k-1}$ if $j_{k-1}>n^{1/4}$, and let $j_k=j_{k-1}-1$ otherwise. Clearly, $j_k$ is a decreasing function of $k$. Let $\tau_1$ be the smallest natural number $k$ such that $j_k\le n^{1/4}$. Let $\tau$ be the natural number $k$ such that $j_k=0$. It is easy to check that $\tau_1=O(\log n)$ and $\tau=O(n^{1/4})$.

We use a clean-up algorithm, which runs in iterations. The $k$-th iteration repeatedly absorbs $j_{k-1}-j_k$ vertices into $P$, leaving $j_k$ unsaturated vertices (vertices that have not been added to $P$) in the end. The $k$-th iteration of the clean-up algorithm works as follows.

\begin{itemize}
    \item[(i)] ({\em Initialising}): Uncolour all vertices in the graph;
    \item[(ii)] ({\em Building reservoir}): Let $m_k:=\sqrt{\eps}(1/2)^{k/2}n$ for $k\le \tau_1$ and $m_k:=\sqrt{n}$ if $\tau_1<k\le\tau$. Add $m_k$ semi-random edges as follows. If $u_t$ lands on an unsaturated vertex, a red vertex, or a neighbour of a red vertex in $P$, then let $v_t$ be chosen arbitrarily. The edge $u_tv_t$ will not be used in our construction. Otherwise, colour $u_t$ red and choose an arbitrary $v_t$ among those unsaturated vertices with the minimum  number of red neighbours. Colour $u_tv_t$ red. Note that each red vertex is adjacent to exactly one red edge; 
    \item[(iii)] ({\em Absorbing via path augmentations}): Add semi-random edges as follows. Suppose that $u_t$ lands on $P$ and at least one neighbour of $u_t$ on $P$ is red. (Otherwise, $v_t$ is chosen arbitrarily, and this edge will not be used in our construction.) Let $x$ be such a red vertex (if $u_t$ has two neighbours on $P$ that are red, then select one of them arbitrarily). Let $y$ by the neighbour of $x$ such that $xy$ is red, and let $v_t=y$. Extend $P$ by deleting the edge $xu_t$ and adding the edges $xy$ and $yu_t$. Uncolour all red edges incident to $y$ and all red neighbours of $y$ (which, of course, includes vertex $x$). 
\end{itemize}

\noindent 
Notice that, in each iteration, $m_k\ge n^{1/2}$. Indeed, this is true for $\tau_1<k\le\tau$. On the other hand, if $k\le \tau_1$, then $j_k = \eps n (1/2)^k$ and so $m_k = \sqrt{n j_k} \ge \sqrt{n}$ (in fact, $m_k = \Omega(n^{5/8})$).

Let $T_k$ denote the length of the $k$-th iteration of the clean-up algorithm. It remains to prove that a.a.s.\ $\sum_{k\le \tau} T_k = O(\sqrt{\eps} n + n^{3/4}\log^2 n)$. Let $R_k$ be the number of red vertices obtained after step (ii) of iteration $k$. Clearly, $R_k\le m_k$.  On the other hand, each $u_t$ is coloured red with probability at least $1-j_{k-1}/n-3m_k/n \ge 1 - \eps - 3\sqrt{\eps} \ge 0.95$. Hence, $R_k$ can be stochastically lower bounded by the binomial random variable $\Bin(m_k, 0.95)$. By the Chernoff bound, with probability at least $1-n^{-1}$, $R_k \ge 0.9 m_k$, as $m_k\ge n^{1/2}$. 

First, we consider iterations $k\le \tau_1$. Let $\tilde R_k$ be the number of red vertices at the end of step~(iii). Note that the minimum degree property of step~(ii) ensures that each unsaturated vertex is adjacent to at most $R_k/j_{k-1}+1 \le m_k/j_{k-1}+1$ red vertices. Moreover, exactly $j_{k-1}-j_k=(1/2)j_{k-1}$ vertices are absorbed in step (iii). As a result, 
\[
\tilde R_k\ge R_k-\left(\frac{m_k}{j_{k-1}}+1\right) \cdot \frac{j_{k-1}}{2} \ge 0.9 m_k-\frac{m_k}{2}-\frac{j_{k-1}}{2} \ge 0.3m_k,
\]
as $j_{k-1} = 2 j_k \le 2 \sqrt{\eps} m_k \le 0.1 m_k$. It follows that throughout step (iii), there are at least $0.3 m_k$ red vertices. Thus, for each semi-random edge added to the graph during step (iii), the probability that a path extension can be performed is at least $0.3m_k/n=0.3\sqrt{\eps}(1/2)^{k/2}$. Again, by the Chernoff bound, with probability at least $1-n^{-1}$, the number of semi-random edges added in step (iii) is at most
\[
2(j_{k-1}-j_k) \cdot \frac{2^{k/2}}{0.3\sqrt{\eps}} \le 7\sqrt{\eps}(1/2)^{k/2} n.
\]
Combining the number of semi-random edges added in step (ii), it follows that with probability at least $1-n^{-1}$, $T_k\le m_k+7\sqrt{\eps}(1/2)^{k/2} n = 8\sqrt{\eps}(1/2)^{k/2} n$.

Next, consider iterations $\tau_1<k\le \tau$. In each iteration, exactly one unsaturated vertex gets absorbed. The number of semi-random edges added in step (ii) is $m_k=n^{1/2}$. We have argued above that with probability at least $1-n^{-1}$, $R_k\ge 0.9m_k$. Thus, for each semi-random edge added to the graph, the probability that a path extension can be performed is at least $0.9m_k/n=0.9n^{-1/2}$. By the Chernoff bound, with probability at least $1-n^{-1}$, the number of semi-random edges added in step (iii) is at most $n^{1/2}\log^2 n$. Thus, with probability at least $1-n^{-1}$, $T_k\le n^{1/2}+n^{1/2}\log^2 n \le 2n^{1/2}\log^2 n$. 

By taking the union bound over all $k\le\tau$, since $\tau = O(n^{1/4})$, it follows that a.a.s.\ 
\[
\sum_{k\le \tau}T_k \le \sum_{k\le \tau_1} 8\sqrt{\eps}(1/2)^{k/2} n + \sum_{\tau_1<k\le \tau} 2n^{1/2}\log^2 n 
=O(\sqrt{\eps}n + n^{3/4}\log^2 n). 
\]

We have shown that a.a.s.\ by adding $O(\sqrt{\eps}n + n^{3/4}\log^2 n)$ additional semi-random edges we can construct a Hamiltonian path $P$. To complete the job and turn it into a Hamiltonian cycle, let $u$ and $v$ denote the left and, respectively, the right endpoint of $P$. We proceed in two stages. In the first stage, add $n^{1/2}\log n$ semi-random edges $u_tv_t$ where $v_t$ is always $u$, discarding any multiple edges that could possibly be created. For each such semi-random edge $u_tu$, colour the left neighbour of $u_t$ on $P$ blue. In the second stage, add $n^{1/2}\log n$ semi-random edges $u_tv_t$ where $v_t$ is always $v$. Suppose that some $u_t=x$ is blue in the second stage. Then, a Hamiltonian cycle is obtained by deleting $xy$ from $P$ and adding the edges $xv$ and $uy$, where $y$ is the right neighbour of $x$ on $P$. By the Chernoff bound, a.a.s.\ a semi-random edge added during the second stage hits a blue vertex, completing the proof. 
\end{proof}

By setting $N=0$ we immediately get an algorithm which a.a.s.\ constructs a Hamiltonian cycle in $\hat\alpha n$ steps, where $\hat\alpha \le 1.84887$. To obtain the better bound in Theorem~\ref{thm:main_upper_bound}, we set $N=100$, and the execution of~\ref{alg:degree_greedy} will be analysed in the next subsection.

\begin{theorem}\label{thm:warm_up}
$C_{\mathtt{HAM}} \le \hat{\alpha} \le 1.84887$, where $\hat{\alpha}$ is defined in~(\ref{alpha_star}) with initial conditions for~(\ref{ode1})--(\ref{ode3}) set by $x(0)=y(0)=\ell_1(0)=\ell_2(0)=0$.
\end{theorem}

\begin{proof} 
This follows from Proposition~\ref{p:boundary}, the numerical value of $\alpha^*$, and Lemma~\ref{lem:clean_up}.
\end{proof}

\subsection{A Degree-Greedy Algorithm} \label{sec:degree_greedy}

Let us suppose that after $t \ge 0$ steps, we have constructed the graph $G_t$  which contains the path $P_t$ and a collection of vertex disjoint edges $\scr{Y}_t$ where $V(\scr{Y}_t) \subseteq [n] \setminus V(P_t)$. We refer to $V(P_t)$ (respectively, $[n] \setminus V(P_t)$) as the \textit{saturated} (respectively, \textit{unsaturated}) vertices of $[n]$.

As before, our algorithm uses path augmentations, and we colour the edges and vertices of $G_t$ to help keep track of when these augmentations can be made. We now use two colours, namely red and blue, to distinguish between edges which are added randomly (red) and greedily (blue). Our blue edges will be chosen so as to minimize the number of blue edges destroyed by path augmentations in future rounds. 

We say that $x \in V(P_t)$ is \textit{blue}, provided it is adjacent to a single blue edge of $G_t$, and no red edge. Similarly, $x \in V(P_t)$ is \textit{red}, provided it is adjacent to a single red edge of $G_t$, and no blue edge. Finally, we say that $x \in V(P_t)$ is \textit{magenta (mixed)}, provided it is adjacent to a single red edge, and a single blue edge. We denote the blue vertices, red vertices, and magenta (mixed) vertices by $\scr{B}_t, \scr{R}_t$ and $\scr{M}_t$, respectively, and define $\scr{L}_t := \scr{B}_t \cup \scr{R}_t \cup \scr{M}_t$ to be the \textit{coloured} vertices. By definition, $\scr{B}_t, \scr{R}_t$ and $\scr{M}_t$ are disjoint. It will be convenient to once again define $U_t$ as the vertices not in $P_t$ or any edge of $\scr{Y}_t$. Finally, we maintain a set of \textit{permissible} vertices $\scr{Q}_t$ which indicate which vertices of the path are allowed to be coloured blue. Specifically, using the same reasoning as before, we ensure the following:
\begin{enumerate}
    \item[(i)] $|\scr{Q}_t| = \max\{|V(P_t)| - 5|\scr{L}_t|, 0\}$. \label{eqn:greedy_permissible_size}
    \item[(ii)] If $\scr{L}_t \neq \emptyset$, then each $x \in \scr{Q}_t$ satisfies $d_{P_t}(x,\scr{L}_t) \ge 3$. \label{eqn:greedy_permissible_distance}
\end{enumerate}

Upon the arrival of $u_\tp$, there are six main cases our algorithm must handle. The first three cases involve extending $P_t$ or $\scr{Y}_t$, whereas the latter three describe how to add edges so that the path can be extended in later rounds.
\begin{enumerate}
    \item If $u_\tp$ lands in $U_t$ and $|U_t| \ge 2$, then choose $v_\tp$ u.a.r. amongst $U_t \setminus \{u_\tp\}$ and extend $\scr{Y}_t$.
    \item If $u_\tp$ lands in $V(\scr{Y}_t)$, then greedily extend $P_{t}$. \label{eqn:greedy_extend}
    \item If $u_\tp$ lands at path distance one from $x \in \scr{L}_t$, then augment $P_t$ via a coloured edge of $x$, where a blue edge is taken over a red edge if possible. \label{eqn:path_augment}
    \item If $u_\tp$ lands in $\scr{Q}_t$, then choose $v_\tp$ u.a.r.\ amongst those vertices of $U_t$ with \textit{minimum} blue degree. The edge $u_\tp v_\tp$ is then coloured blue, and a single blue vertex is created. \label{eqn:blue_vertex} 
    \item If $u_\tp$ lands in $\scr{R}_t$, then choose $v_\tp$ u.a.r.\ amongst those vertices of $U_t$ with minimum blue degree. The edge $u_\tp v_\tp$ is then coloured blue, and a single red vertex is converted to a magenta (mixed) vertex. \label{eqn:red_vertex} 
    \item If $u_\tp$ lands in $\scr{B}_t$, then choose $v_\tp$ u.a.r.\ amongst $U_t$ and colour $u_\tp v_\tp$ red. This case converts a blue vertex to a magenta vertex. \label{eqn:magenta_vertex}
\end{enumerate}
In all the remaining cases, we choose $v_\tp$ arbitrarily, and interpret the algorithm as \textit{passing} on the round. As in \ref{alg:fully_randomized}, we ensure that all of the algorithm's coloured vertices are at path distance at least $3$ from each other, and we define a coloured vertex to be \textit{well-spaced} in the same way. Step $\tp$ of the algorithm when $u_\tp$ is drawn u.a.r.\ from $[n]$ is formally described by the \ref{alg:degree_greedy} algorithm. We describe how the algorithm chooses $v_\tp$, how it constructs $P_\tp$ and $\scr{Y}_\tp$, and how it adjusts the colours of $G_\tp$, thus updating $\scr{B}_t, \scr{M}_t$ and $\scr{R}_t$. 
\begin{varalgorithm}{$\mathtt{DegreeGreedy}$}
\caption{Step $\tp$} 
\label{alg:degree_greedy}
\begin{algorithmic}[1]
\If{$u_\tp \in U_t$ and $|U_{t}| \ge 2$}     \Comment{greedily extend $\scr{Y}_t$}
\State Choose $v_\tp$ u.a.r. from $U_t \setminus V( \scr{Y}_t) \cup \{u_\tp\}$.
\State Set $P_\tp = P_t$ and $\YY_\tp=\YY_t\cup \{u_\tp v_\tp\}$.  
\ElsIf{$u_\tp y \in \scr{Y}_t$ for some $y$} \Comment{greedily extend the path}
\State Let $v_\tp$ be an arbitrarily chosen endpoint of $P_t$. 
\State Update $P_\tp$ from $P_t$ by adding edges $u_\tp v_\tp$ and $u_\tp y$. 
\State Set $\YY_\tp=\YY_{t}\setminus \{u_\tp y\}$.
\State Uncolour all of the edges adjacent to $u_\tp$ or $y$. \label{line:unsaturated_uncolour}

\ElsIf{$d(u_\tp, \scr{L}_t) =1$}     \Comment{path augment via coloured vertices}
\State Let $x \in \scr{L}_t$ be the (unique) coloured vertex adjacent to $u_\tp$

\If{$x$ is red}
\State Denote $x y \in E(G_t)$ the red edge of $x$.
\Else{}         \Comment{$x$ is blue or magenta}
\State Denote $x y \in E(G_t)$ the blue edge of $x$.
\EndIf
\If{$y \in U_t$}
\State Set $v_\tp = y$
\State Update $P_\tp$ from $P_{t}$ by adding edges $u_\tp v_\tp, v_\tp x$ and removing edge $u_\tp x$.
\State Set $\scr{Y}_\tp = \scr{Y}_t$.
\ElsIf{$yy' \in \scr{Y}_t$}
\State Set $v_\tp = y'$
\State Update $P_\tp$ from $P_t$ by adding edges $u_\tp v_\tp, v_\tp y, y x$ and removing edge $u_\tp x$.
\State Set $\scr{Y}_\tp = \scr{Y}_t \setminus \{yy'\}$
\EndIf

\State Uncolour all of the edges adjacent to $y$ (as well as $y'$ if applicable). \label{line:saturated_uncolour}

\Else                                         \Comment{construct coloured vertices or pass}
\If{$u_\tp \in \scr{Q}_t \cup \scr{R}_t$}   
\State Choose $v_\tp$ u.a.r.\ from the vertices of $U_t$ of minimum blue degree.
\State Colour $u_\tp v_\tp$ blue.       \Comment{create a blue or magenta vertex} 

\ElsIf{$u_\tp \in \scr{B}_t$}
\State Choose $v_\tp$ u.a.r.\ from $U_t$.
\State Colour the edge $u_\tp v_\tp$ red.  \Comment{create a magenta vertex}

\Else{} \Comment{pass on using edge $u_\tp v_\tp$}
\State Choose $v_\tp$ arbitrarily \ from $[n]$. 
\EndIf

\State Set $P_\tp = P_t, \scr{Y}_\tp = \scr{Y}_t$.

\EndIf
\State Update $\scr{Q}_\tp$ if needed, such that $|\scr{Q}_\tp| = |V(P_\tp)| - 5|\scr{L}_\tp|$.    \Comment{update permissible vertices}
\end{algorithmic}
\end{varalgorithm}
Note that red vertices are only created when the blue edges of magenta vertices are uncoloured as a side effect of path extensions and augmentations (see lines \eqref{line:unsaturated_uncolour} and \eqref{line:saturated_uncolour} of \ref{alg:degree_greedy}).

For each $t \ge 0$, define the random variables $X(t) := |V(P_t)|$, $B(t):= |\scr{B}_t|$, $R(t) := |\scr{R}_t|$, $M(t) := |\scr{M}_t|$, $L(t):=|\scr{L}_t|=B(t)+R(t)+M(t)$, and $Y(t):= |V(\scr{Y}_t)|$. For each $q \ge 0$  define $D_{q}(t)$ to be the number of unsaturated vertices adjacent to precisely $q$ blue edges.  We define the stopping time $\tau_{q}$ to be the smallest $t \ge 0$ such that $D_{j}(t) = 0$ for all $j<q$, and $D_{q}(t) > 0$. It is easy to check that $\tau_q$ is well-defined and is non-decreasing in $q$. By definition, $\tau_{0}=0$. Let us refer to \textit{phase} $q$ as those $t$ such that $\tau_{q-1} \le t < \tau_{q}$. Observe that during phase $q$, each unsaturated vertex (i.e., vertex of $[n] \setminus V(P_t)$) has blue degree $q-1$ or $q$.

\subsection{Analyzing phase $q$}

Suppose that $\tau_{q-1} \le t < \tau_{q}$. It will be convenient to denote $D(t):=D_{q-1}(t)$. Given $k_1, k_2 \ge 0$, we say that $y \in [n] \setminus V(P_t)$ is of \textit{type} $(k_1,k_2)$, provided it is adjacent to $k_1$ blue edges within $\scr{B}_t$ and $k_2$ blue edges within $\scr{M}_t$. Similarly, $x \in \scr{B}_t \cup \scr{M}_t$ is of type $(k_1,k_2)$, provided its (unique) \textit{blue} edge connects to a vertex of type $(k_1,k_2)$. We denote the number of unsaturated vertices of type $(k_1,k_2)$ by $C_{k_1,k_2}(t)$, the blue vertices of type $(k_1,k_2)$ by $B_{k_1,k_2}(t)$, and the magenta (mixed) vertices of type $(k_1,k_2)$ by $M_{k_1,k_2}(t)$. Observe that $B_{k_1,k_2}(t) = k_1 \cdot C_{k_1,k_2}(t)$ and $M_{k_1,k_2}(t) = k_2 \cdot C_{k_1,k_2}(t)$. Moreover, $D_{j}(t) = \sum_{\substack{k_1, k_2: \\ k_{1} + k_{2} = j}} C_{k_1,k_2}(t)$.

In \Cref{sec:inductive_functions}, we inductively define the functions $x, r, y$ and $c_{k_1,k_2}$ for $k_1 + k_2 \ge 0$, as well as a constant $\sigma_q \ge 0$, such that the following lemma holds:

\begin{lemma}\label{lem:inductive}
A.a.s.\ $\tau_{q} \sim \sigma_{q} n$ for every $0 \le q\le N$.\footnote{For functions $f = f(n)$ and $g=g(n)$, $f \sim g$ is shorthand for $f =(1+ o(1)) g$.} Moreover, at step $\tau_{q}$, a.a.s.
\begin{eqnarray*}
&&X(\tau_{q}) \sim x(\sigma_{q})n, \quad R(\tau_{q}) \sim r(\sigma_{q})n, \quad Y(\tau_q) \sim y(\sigma_q)n, \\
&&C_{k_1,k_2}(\tau_{q}) \sim c_{k_1,k_2}(\sigma_{q})n \quad \mbox{for all $(k_1,k_2)$ where $k_1+k_2=q$.}
\end{eqnarray*}
\end{lemma}

Although the method in the proof of Lemma~\ref{lem:inductive} is similar to that of Lemmas~\ref{lem:lipschitz_randomized},~\ref{lem:randomized_expected_differences},~\ref{lem1:concentration_random} and Proposition~\ref{p:boundary}, the analysis is much more intricate and involved. We postpone the proof until afterwards, and first complete the proof of Theorem~\ref{thm:main_upper_bound}.

\subsection{Proving \Cref{thm:main_upper_bound} assuming \Cref{lem:inductive}} \label{sec:proof_main}

\begin{proof}[Proof of Theorem~\ref{thm:main_upper_bound}] 
Set $N=100$. By Lemma~\ref{lem:inductive}, the execution of \ref{alg:degree_greedy} ends at some step $t_0\sim \sigma_N n$. Moreover, $X(t_0)\sim x(\sigma_N)n$, $Y(t_0) \sim y(\sigma_N)n$, $R(t_0) \sim r(\sigma_N)n$ and $M(t_0) \sim m(\sigma_N)n$. Numerical computations show that $\sigma_N\approx 1.80249$. Let $P$ be the path constructed after the first $t_0$ rounds, $\scr{Y}$ be the edges of $\scr{Y}_{t_0}$, and $\scr{E}$ be the red edges adjacent to the vertices of $\scr{M}_{t_0} \cup \scr{R}_{t_0}$.  By the definition of \ref{alg:degree_greedy}, in particular by the way that $\YY_t$ is extended, and the way that the red edges are created, $\YY$ has the uniform distribution over all possible $|\YY|$ pairs over vertices that are not on the path $P$; and for each red edge, its end that is not on the path $P$ is also uniformly distributed. Thus, $(P, \scr{Y},\scr{E})$ has the distribution required for the analysis of \ref{alg:fully_randomized}. Let 
\begin{align*}
\hat{x}&:=x(\sigma_N) \approx 0.99991\\
\hat{y}&:=y(\sigma_N)\approx 0.0000029724 \\
\hat{\ell_1}&:=m(\sigma_N)+r(\sigma_N)\approx 0.00019429.  
\end{align*}
Then, $\hat{x}$, $\hat{y}$, and $\hat{\ell}_{1}$ satisfy $|P| \sim \hat{x}n$, $|\scr{Y}| \sim \hat{y}n$ and $|\scr{E}| \sim \hat{\ell}_1 n$. (The final equation holds since each vertex of $\scr{M}_{t_0} \cup \scr{R}_{t_0}$ is adjacent to one red edge.)

We next execute \ref{alg:fully_randomized} with initial input $(P, \scr{Y}, \scr{E})$. Let $\alpha^*$ be as defined in (\ref{alpha_star}) where the initial conditions to the differential equations (\ref{ode1})--(\ref{ode3}) are set by $x(0)=\hat{x}$, $\ell_1(0)= \hat{\ell}_1$ and $\ell_{2}(0)=0$. Numerical computations show that $\alpha^*\approx 0.014468$. By Proposition~\ref{p:boundary} and the fact that  $\alpha^*<3$, the execution of the first two stages (\ref{alg:degree_greedy} and \ref{alg:fully_randomized}) finishes at some step $(\sigma_N + \alpha^* + o(1))n \le 1.81696 n$, and the number of unsaturated vertices remaining is $o(n)$. Finally, the clean-up algorithm constructs a Hamiltonian cycle with an additional $o(n)$ steps by Lemma~\ref{lem:clean_up}. The theorem follows.
\end{proof}

\subsection{Proving Lemma \ref{lem:inductive}} \label{sec:inductive_functions}

We once again must first argue that our random variables cannot change drastically in one round during phase $q$.
\begin{lemma}[Lipschitz Condition -- \ref{alg:degree_greedy}] \label{lem:boundedness_greedy}
If $|\Delta C(t)| := \max_{\substack{k_1, k_2 \in \Nn \cup \{0\}: \\ k_1 + k_2 \in \{q-1,q\}}} |\Delta C_{k_1,k_2}(t)|$, then with probability $1- O(n^{-1})$,
$$\max\{ |\Delta X(t)|, |\Delta C(t)|, |\Delta R(t)|, |\Delta Y(t)| \} = O(\log n)$$
for all $\tau_{q-1} \le t < \tau_{q}$ with $n - X(t) = \Omega(n)$. 
\end{lemma}
\begin{proof}
Since $q \le N$ is a constant which does not depend on $n$, we can apply the same argument to bound the red edges of each $\Delta C_{k_1,k_2}(t)$ as in Lemma \ref{lem:lipschitz_randomized}, and then union bound over all $k_1,k_2 \ge 0$ such that $k_1 + k_2 \in \{q-1,q\}$.
\end{proof}

Let $H_t$ denote the history of the above random variables during the first $t$ rounds. We now state the conditional expected differences of our random variables, where we assume that $\tau_{q-1} \le t < \tau_{q}$ is such that $n - X(t) = \Omega(n)$. It will be convenient to define auxiliary random variables $A(t):= 2 Y(t)/(1-X(t))$ and $\Gamma(t):=1 + Y(t)/(1-X(t))$. Then,
\begin{eqnarray}\label{eqn:degree_x_change}
\E[\Delta X(t)\mid H_t]&=& \frac{2 Y(t)}{n}  + \frac{2 L(t)}{n}\ \Gamma(t) +O(1/n)
\end{eqnarray}
and
\begin{equation} \label{eqn:degree_y_change}
\E[ \Delta Y(t) \mid H_t] = -\frac{2 Y(t)}{n} + 2 \left(1 -\frac{X(t) -Y(t)}{n} \right) - \frac{2 L(t)}{n} \ A(t).
\end{equation}
We omit the proofs of \eqref{eqn:degree_x_change} and \eqref{eqn:degree_y_change}, as the derivation is the same as the analogous equations of \Cref{lem:randomized_expected_differences}. For the remaining random variables, we state the expected differences and derive them afterwards.

First, consider $\Delta R(t)$:
\begin{eqnarray} \label{eqn:degree_r_change}
\E[\Delta R(t) \mid H_t] &=& \frac{Y(t)}{n} \left(\frac{2M(t)}{n-X(t)} -\frac{2R(t)}{n-X(t)} \right) -\frac{2 (B(t) +M(t)) }{n} \frac{R(t) \Gamma(t)}{n - X(t)}  \nonumber \\
&&+ \sum_{\substack{j,h: \\ j+ h \in \{q-1,q\}}} \frac{2 B_{j,h}(t)}{n} \left( h + \frac{Y(t)}{n-X(t)} \frac{M(t)}{n - X(t)} \right) \nonumber \\
&&+ \sum_{\substack{j,h: \\ j+ h \in \{q-1,q\}}} \frac{2 M_{j,h}(t)}{n} \left( h + \frac{Y(t)}{n-X(t)} \frac{M(t)}{n - X(t)} \right)  \nonumber \\
&&-\frac{2 R(t) }{n} \left(1 + \frac{R(t) \Gamma(t)}{n - X(t)} \right) + \frac{2 R(t)}{n} \frac{M(t) \Gamma(t)}{n-X(t)} - \frac{R(t)}{n} +O(1/n).
\end{eqnarray}
If $k_1 + k_2 = q-1$, then $\Delta C_{k_1,k_2}(t)$ satisfies: 
\begin{eqnarray} \label{eqn:degree_c_min_change}
\E[ \Delta C_{k_1,k_2}(t) \mid H_t] 
&=& \frac{Y(t)}{n} \left( \frac{2 M_{k_1 -1,k_2 +1}(t)}{n - X(t)} \cdot \bm{1}_{k_1 >0} -\frac{2 C_{k_1,k_2}(t)}{n - X(t)} - \frac{2 M_{k_1,k_2}(t)}{n - X(t)} \right) \nonumber \\
&&+ \frac{2 (B(t) +M(t))}{n} \left( \frac{M_{k_1 -1,k_2 +1}(t)}{n - X(t)} \cdot \bm{1}_{k_1 > 0} - \frac{M_{k_1,k_2}(t)}{n - X(t)} \right) \Gamma(t)  \nonumber \\
&&- \frac{2(B(t) +M(t))}{n} \frac{A(t)}{2} \frac{C_{k_1,k_2}(t)}{n-X(t)} - \frac{2 B_{k_1,k_2}(t) }{n} - \frac{2  M_{k_1,k_2}(t)}{n} \nonumber\\
&&+ \frac{2 R(t)}{n } \left( \frac{M_{k_1 -1,k_2 +1}(t)}{n - X(t)} \cdot \bm{1}_{k_1 >0} - \frac{M_{k_1,k_2}(t)}{n - X(t)} - \frac{C_{k_1,k_2}(t)}{n - X(t)} \right) \Gamma(t) \nonumber \\
&&-\frac{(X(t)- 5 L(t))}{n} \frac{C_{k_1,k_2}(t)}{D(t)}  -\frac{R(t)}{n} \frac{C_{k_1,k_2}(t)}{D(t)} \nonumber \\
&&+ \frac{B_{k_1 +1,k_2 -1}(t)}{n} \cdot \bm{1}_{k_2 > 0} -\frac{B_{k_1,k_2}(t)}{n} + O(1/n).
\end{eqnarray}
When $k_1 + k_2 =q$, two terms from the above expression are modified slightly, and have their signs reversed:
\begin{eqnarray} \label{eqn:degree_c_max_change}
\E[ \Delta C_{k_1,k_2}(t) \mid H_t] 
&=& \frac{Y(t)}{n} \left( \frac{2 M_{k_1 -1,k_2 +1}(t)}{n - X(t)} \cdot \bm{1}_{k_1 >0} -\frac{2 C_{k_1,k_2}(t)}{n - X(t)} - \frac{2 M_{k_1,k_2}(t)}{n - X(t)} \right) \nonumber \\
&&+ \frac{2 (B(t) +M(t))}{n} \left( \frac{M_{k_1 -1,k_2 +1}(t)}{n - X(t)} \cdot \bm{1}_{k_1 > 0} - \frac{M_{k_1,k_2}(t)}{n - X(t)} \right) \Gamma(t) \nonumber \\
&&- \frac{2(B(t) +M(t))}{n} \frac{A(t)}{2} \frac{C_{k_1,k_2}(t)}{n-X(t)} - \frac{2 B_{k_1,k_2}(t) }{n} - \frac{2  M_{k_1,k_2}(t)}{n} \nonumber\\
&&+ \frac{2 R(t)}{n } \left( \frac{M_{k_1 -1,k_2 +1}(t)}{n - X(t)} \cdot \bm{1}_{k_1 >0} - \frac{M_{k_1,k_2}(t)}{n - X(t)} - \frac{C_{k_1,k_2}(t)}{n - X(t)} \right) \Gamma(t) \nonumber \\
&&+\frac{(X(t)- 5 L(t))}{n} \frac{C_{k_1-1,k_2}(t)}{D(t)} +\frac{R(t)}{n} \frac{C_{k_1,k_2-1}(t)}{D(t)}  \nonumber  \\
&&+ \frac{B_{k_1 +1,k_2 -1}(t)}{n} \cdot \bm{1}_{k_2 > 0} -\frac{B_{k_1,k_2}(t)}{n} + O(1/n).
\end{eqnarray}

In order to prove the expected differences, we analyze the expected values of the random variables $\Delta R(t)$ and $\Delta C_{k_1,k_2}(t)$ when $u_\tp$ lands in a subset $A \subseteq [n]$ for a number of choices of $A$. More precisely, we derive tables for $\mb{E}[ \Delta R(t) \cdot \bm{1}_{u_\tp \in A} \mid H_t]$ and $\mb{E}[ \Delta C_{k_1,k_2}(t) \cdot \bm{1}_{u_\tp \in A} \mid H_t]$ when $A \subseteq [n]$ varies across a number of subsets. Since these are disjoint subsets of $[n]$, and the random variables are $0$ if $u_\tp$ lands outside of these subsets, we can sum the second column entries to get the claimed expected differences. Note that the entries of our tables do \textit{not} contain the often necessary $O(1/n)$ term. 

In all our below explanations, we abuse notation and simultaneously identify our random variables as sets (i.e., $C_{k_1,k_2}(t)$ denotes the set of unsaturated vertices of type $(k_1,k_2)$ after $t$ steps).
\begin{table}[H]
\caption{Expected Changes to $\Delta R(t)$.}\label{table:red_changes}
\begin{tabular}{|c|c|}
\hline
$A \subseteq [n]$  & $\mb{E}[ \Delta R(t) \cdot \bm{1}_{u_\tp \in A} \mid H_t]$ \\
\hline
$U_t$ & $0$ \\
$V(\scr{Y}_t)$ & $\frac{Y(t)}{n} \left(\frac{2M(t)}{n-X(t)} -\frac{2R(t)}{n-X(t)} \right)$ \\
Path distance $1$ from $\scr{B}_t $  & $-\frac{2 B(t) }{n} \frac{R(t) \Gamma(t)}{n - X(t)} + \sum_{\substack{j,h: \\ j+ h \in \{q-1,q\}}} \frac{2 B_{j,h}(t)}{n} \left( h + \frac{Y(t)}{n-X(t)} \frac{M(t)}{n - X(t)} \right) $  \\ 
Path distance $1$ from $ \scr{M}_t$  & $-\frac{2 M(t) }{n} \frac{R(t) \Gamma(t)}{n - X(t)} + \sum_{\substack{j,h: \\ j+ h \in \{q-1,q\}}} \frac{2 M_{j,h}(t)}{n} \left( h + \frac{Y(t)}{n-X(t)} \frac{M(t)}{n - X(t)} \right) $   \\ 
Path distance $1$ from $\scr{R}_t$  & $-\frac{2 R(t) }{n} \left(1 + \frac{R(t) \Gamma(t)}{n - X(t)} \right) + \frac{2 R(t)}{n} \frac{M(t) \Gamma(t)}{n-X(t)}$ \\
$\scr{R}_t$ & $\frac{-R(t)}{n}$ \\
\hline
\end{tabular}
\end{table}

\begin{proof}[Proof Sketch of Table \ref{table:red_changes}]
We provide complete proofs only of row entries $2$ and $3$, as the remaining entries follow similarly.

In order to see the second entry, expose the blue edges adjacent to $\scr{M}_t$, and the red edges adjacent to $\scr{R}_t$. If we then fix an arbitrary edge $y y' \in \scr{Y}_t$, we know that it is distributed u.a.r. amongst $\binom{[n] \setminus V(P_t)}{2}$. Thus, in expectation there are $2 M(t)/(n-X(t))$ blue edges adjacent to the vertices of $yy'$. Now, if $u_\tp$ lands on $y$ or $y'$, then a path augmentation is made, and these blue edges are destroyed, thus converting $2 M(t)/(n-X(t))$ blue vertices to red vertices in expectation. Since this event occurs with probability $Y(t)/n$, this accounts for the $\frac{Y(t)}{n} \frac{2M(t)}{n-X(t)}$ term. An analogous argument applies to the $\frac{-Y(t)}{n} \frac{2R(t)}{n-X(t)}$ term.

Consider now the third row entry, where we shall first derive the $-\frac{2 B(t) }{n} \frac{R(t)}{n - X(t)} \left(1 + \frac{Y(t)}{n-X(t)} \right)$ term. Suppose that $b$ is a blue vertex, with blue edge $b x$. Observe that there are $R(t)/(n-X(t))$ red edges adjacent to $x$ in expectation. Thus, $\frac{R(t)}{n - X(t)}$ red edges are destroyed in expectation if $u_\tp$ lands next to a blue vertex. Since $u_\tp$ lands next to a blue vertex with probability $2 B(t)/n$, the $-\frac{2 B(t) }{n} \frac{R(t)}{n - X(t)}$ term follows. The $-\frac{2 B(t) }{n} \frac{R(t)}{n - X(t)}\frac{Y(t)}{1-X(t)}$ term follows by observing that $x =y$ for some $yy' \in \scr{Y}_t$ with probability $\frac{Y(t)}{1-X(t)}$. When this occurs, by computing the expected number of red edges adjacent to $y'$, an additional $\frac{2 B(t) }{n} \frac{R(t)}{n - X(t)}$ red edges are destroyed in expectation. These two cases account for the $-\frac{2 B(t) }{n} \frac{R(t)}{n - X(t)} \left(1 + \frac{Y(t)}{n-X(t)} \right)$ term.

We now derive the term:
$$
\sum_{\substack{j,h: \\ j+ h \in \{q-1,q\}}} \frac{2 B_{j,h}(t)}{n} \left( h + \frac{Y(t)}{n-X(t)} \frac{M(t)}{n - X(t)}\right).
$$
Suppose that $u_\tp$ lands next to a blue vertex of $B_{j,h}(t)$, which occurs with probability $2 B_{j,h}(t)/n$. Let $b$ be such a vertex, and denote its blue edge by $bx$. Now, by definition, there are $h$ magenta vertices whose blue edge is also incident to $x$. We claim that all $h$ of these magenta vertices will be reclassified as red vertices provided the following event occurs:
\begin{itemize}
    \item All the red edges of these $h$ magenta vertices are \textit{not} adjacent to $x$.
\end{itemize}
The latter occurs with probability $\left(1 - \frac{1}{n - X(t)} \right)^h = 1 - O(1/n)$, since $h$ is a constant, and $n-X(t) = \Omega(n)$. By summing over $j,h \in \{q-1,q\}$, this yields the expression $\sum_{\substack{j,h: \\ j+ h \in \{q-1,q\}}} \frac{2 h B_{j,h}(t)}{n} + O\left(\frac{1}{n} \right),$ and the $\sum_{\substack{j,h: \\ j+ h \in \{q-1,q\}}} \frac{2 B_{j,h}(t)}{n} \frac{Y(t)}{n-X(t)} \frac{M(t)}{n - X(t)}$ term follows similarly.
\end{proof}

Consider now $\Delta C_{k_1,k_2}(t)$, where $k_{1} + k_{2} = q-1$. 
\begin{table}[H]
\caption{Expected Changes to $\Delta C_{k_1,k_2}(t)$ for $k_{1} + k_{2} = q-1$.} \label{table:type_changes_min}
\begin{tabular}{|c|c|} 
\hline
$A \subseteq [n]$  & $\mb{E}[ \Delta C_{k_1,k_2}(t) \cdot \bm{1}_{u_\tp \in A} \mid H_t]$ \\
\hline
$U_t$ & $0$ \\
$V(\scr{Y}_t)$ & $\frac{Y(t)}{n} \left( \frac{2 M_{k_1 -1,k_2 +1}(t)}{n - X(t)} \cdot \bm{1}_{k_1 >0} -\frac{2 C_{k_1,k_2}(t)}{n - X(t)} - \frac{2 M_{k_1,k_2}(t)}{n - X(t)} \right)$  \\
Path distance $1$ from $\scr{B}_t$ & $\frac{2 B(t) }{n} \left( \frac{M_{k_1 -1,k_2 +1}(t)}{n - X(t)} \cdot \bm{1}_{k_1 > 0} - \frac{M_{k_1,k_2}(t)}{n - X(t)} \right) \Gamma(t) - \frac{2B(t)}{n} \frac{A(t)}{2} \frac{C_{k_1,k_2}(t)}{n-X(t)} - \frac{2 B_{k_1,k_2}(t)}{n}$ \\ 
Path distance $1$ from $\scr{M}_t$ & $\frac{2 M(t)}{n} \left( \frac{M_{k_1 -1,k_2 +1}(t)}{n - X(t)} \cdot \bm{1}_{k_1 > 0} - \frac{M_{k_1,k_2}(t)}{n - X(t)} \right) \Gamma(t) - \frac{2M(t)}{n} \frac{A(t)}{2} \frac{C_{k_1,k_2}(t)}{n-X(t)} - \frac{2  M_{k_1,k_2}(t)}{n} $ \\ 
Path distance $1$ from $\scr{R}_t$  &  $\frac{2 R(t)}{n } \left( \frac{M_{k_1 -1,k_2 +1}(t)}{n - X(t)} \cdot \bm{1}_{k_1 >0} - \frac{M_{k_1,k_2}(t)}{n - X(t)} - \frac{C_{k_1,k_2}(t)}{n - X(t)} \right) \Gamma(t)$ \\
$\scr{Q}_t$ & $-\frac{(X(t)- 5 L(t))}{n} \frac{C_{k_1,k_2}(t)}{D(t)}$ \\
$\scr{R}_t$ & $-\frac{R(t)}{n} \frac{C_{k_1,k_2}(t)}{D(t)}$ \\
$\scr{B}_t$ & $-\frac{B_{k_1,k_2}(t)}{n} + \frac{B_{k_1 +1,k_2 -1}(t)}{n} \cdot \bm{1}_{k_2 > 0}$ \\
\hline
\end{tabular}
\end{table}

\begin{proof}[Proof Sketch of Table \ref{table:type_changes_min}]
Assume that $k_1, k_2$ are both non-zero, as this is the most involved case. We provide complete  proofs of row entries $2,5$ and $6$.

We begin with row entry $2$. Observe that $u_\tp$ lands on a vertex of an edge of $\scr{Y}_t$, say $yy'$, with probability $Y(t)/n$. At this point, a path augmentation occurs, and $yy'$ is added to the current path $P_t$, thus destroying the red and blue edges adjacent to $y$ and $y'$. We claim that $$\mb{E}[ \Delta C_{k_1,k_2}(t) \mid H_t, \{u_\tp \in V(\scr{Y}_t)\}] =\frac{2 M_{k_1 -1,k_2 +1}(t)}{n - X(t)} \cdot \bm{1}_{k_1 >0} -\frac{2 C_{k_1,k_2}(t)}{n - X(t)} - \frac{2 M_{k_1,k_2}(t)}{n - X(t)}.$$

Let us focus on the $-\frac{2 C_{k_1,k_2}(t)}{n - X(t)} - \frac{2 M_{k_1,k_2}(t)}{n - X(t)}$ term. In order to see this, fix an unsaturated vertex $c$ of type $(k_1,k_2)$. We shall prove that $c$ is \textit{destroyed} (i.e., removed from $C_{k_1,k_2}(t)$) with probability $(k_2 + 1)/(n - X(t)) + O(1/n^2)$, conditional on $H_t$ and $\{u_\tp \in V(\scr{Y}_t)\}$. By summing over all $c \in C_{k_1,k_2}(t)$ and using the fact that $k_2 \cdot C_{k_1,k_2}(t) = M_{k_1,k_2}(t)$, this yields the $-\frac{2C_{k_1,k_2}(t)}{n -X(t)} - \frac{2M_{k_1,k_2}(t)}{n -X(t)}$ term. The  $\frac{2M_{k_1 -1,k_2 +1}(t)}{n-X(t)}$ term follows similarly, where reclassifying unsaturated vertices of type $(k_1 -1, k_2 +1)$ to type $(k_1, k_2)$ causes $C_{k_1,k_2}(t)$ to increase.

Let us now prove the above claim regarding $c \in C_{k_1,k_2}(t)$, where we condition on $H_t$ and $u_\tp = y$ for some $yy' \in \scr{Y}_t$ in the below explanations. Suppose that $m_1, \ldots , m_{k_2}$ are the magenta neighbours of $c$ of type $(k_1,k_2 )$. By definition, $c m_i$ is coloured blue, and each $m_i$ also has a red edge $m_i x_i$ for $i=1,\ldots, k_2$. Observe that if either $y$ or $y'$ is equal to $c$, then $c$ will be added to the path (and thus destroyed). Similarly, if $x_i$ is added to the path, then the edge $m_i x_i$ is no longer red. In particular, $m_i$ is converted to a blue vertex, and so the type of $c$ is reclassified as $(k_1 +1, k_2 -1)$. In either case, $c$ is destroyed. Now, $x_1, \ldots ,x_{k_2}$ are distributed u.a.r. and independently amongst $[n] \setminus V(P_t)$, and so the vertices $c, x_1, \ldots ,x_{k_2}$ are distinct with probability $ \prod_{i=1}^{k_2}\left(1 - \frac{i}{n-X(t)}\right) = 1 - O(1/n)$, where we have used the fact that $k_2$ is a constant and $n - X(t) = \Omega(n)$. Moreover, $yy'$ is distributed u.a.r. and independently amongst $\binom{[n] \setminus V(P_t)}{2}$. Thus, conditional on the vertices $c, x_{1}, \ldots ,x_{k_2}$ being distinct, $c$ is destroyed with probability $$ \frac{ 2(k_2 +1)}{n-X(t)} -\frac{ \binom{k_2 +1}{2}}{ \binom{n -X(t)}{2}} =  \frac{2 (k_2 + 1)}{n - X(t)} - O(1/n^2).$$ As such, $c$ is destroyed with the claimed probability of $2(k_2 + 1)/(n - X(t)) + O(1/n^2)$.

Consider row entry $6$. We begin by deriving the expression:
\begin{equation} \label{eqn:dist_1_red_destroy}
\frac{2 R(t)}{n } \left(1 + \frac{Y(t)}{1- X(t)} \right) \left(  -\frac{M_{k_1,k_2}(t)}{n - X(t)} - \frac{C_{k_1,k_2}(t)}{n - X(t)} \right).
\end{equation}
First, condition on the event when $u_\tp$ lands at path distance one from some (red) vertex $r \in \scr{R}_t$. This occurs with probability $2 R(t)/n$. Let $rx$ be the unique red edge of $r$, where $x \in [n] \setminus V(P_t)$. We also condition on the event that $x =y$ for some $y y' \in \scr{Y}_t$, which occurs with probability $Y(t)/(1 - X(t))$. Observe that when these events occur, \ref{alg:degree_greedy} adds $x$ and $y'$ to the path $P_t$ via a path augmentation, and thus destroys the red and blue edges adjacent to $x$ and $y'$.

Fix a vertex $c \in C_{k_1,k_2}(t)$. Conditional on the above events, we claim that $c$ is destroyed with probability $2(k_2 +1)/(n-X(t)) + O(1/n^2)$. To see this, observe that $xy'$ is distributed u.a.r. amongst $\binom{[n] \setminus V(P_t)}{2}$. Thus, the same argument used to derive row entry $2$ applies in this case. By summing over all $c \in C_{k_1,k_2}(t)$ (and multiplying by $2R(t)/n$ and $Y(t)/(1-X(t))$), this yields the term 
$$
\frac{2 R(t)}{n } \frac{Y(t)}{1- X(t)} \left( -\frac{2M_{k_1,k_2}(t)}{n - X(t)} - \frac{2C_{k_1,k_2}(t)}{n - X(t)} \right).
$$
In order to complete the derivation of \eqref{eqn:dist_1_red_destroy}, we consider the event when $x$ is \textit{not} a vertex of an edge of $\scr{Y}_t$. In this case, we only need to account for the red and blue edges destroyed when $x$ is added to the path. This yields the following term:
$$
\frac{2 R(t)}{n } \left( 1-\frac{Y(t)}{1- X(t)} \right) \left( -\frac{M_{k_1,k_2}(t)}{n - X(t)} - \frac{C_{k_1,k_2}(t)}{n - X(t)} \right).
$$
The remaining term $\frac{2 R(t)}{n } \left(1 + \frac{Y(t)}{1- X(t)} \right)  \frac{M_{k_1 -1,k_2 +1}(t)}{n - X(t)} \cdot \bm{1}_{k_1 >0}$ missing from \eqref{eqn:dist_1_red_destroy} follows by a similar argument, where we destroy vertices of $C_{k_1 -1,k_2 +1}(t)$ to create new vertices of $C_{k_1,k_2}(t)$.
 
Let us now consider row entry $6$ when $u_\tp$ lands on a permissible vertex $x \in \scr{Q}_t$. Clearly, this event occurs with probability $|\scr{Q}_t|/n = (X(t) - 5L(t))/n$. In this case, the algorithm chooses $v_\tp$ u.a.r. amongst $D(t)$, the unsaturated vertices of minimum degree $q-1$, and colours the edge $x v_\tp$ blue. Thus, if we fix $c \in C_{k_1,k_2}(t)$, then $c$ will be chosen with probability $1/D(t)$ since $k_1 + k_2 = q-1$. In this case, $c$ gains a blue edge connected to a blue vertex, and thus will be reclassified as type $(k_1 + 1, k_2)$. Thus, each $c \in C_{k_1,k_2}(t)$ will be reclassified with probability $\frac{X(t) - 5L(t)}{n} \frac{1}{D(t)}$. By summing over all $c \in C_{k_1,k_2}(t)$, we get the $-\frac{(X(t)- 5 L(t))}{n} \frac{C_{k_1,k_2}(t)}{D(t)}$ term.
\end{proof}

Finally, when $k_{1} + k_{2} = q$, the expressions in rows $6$ and $7$ are modified slightly.

\begin{table}[H]
\caption{Expected Changes to $\Delta C_{k_1,k_2}(t)$ for $k_{1} + k_{2} = q$} \label{table:type_changes_max}
\begin{tabular}{|c|c|} 
\hline
$A \subseteq [n]$  & $\mb{E}[ \Delta C_{k_1,k_2}(t) \cdot \bm{1}_{u_\tp \in A} \mid H_t]$ \\
\hline
$U_t$ & $0$ \\
$V(\scr{Y}_t)$ & $ \frac{Y(t)}{n} \left( \frac{2 M_{k_1 -1,k_2 +1}(t)}{n - X(t)} \cdot \bm{1}_{k_1 >0} -\frac{2 C_{k_1,k_2}(t)}{n - X(t)} - \frac{2 M_{k_1,k_2}(t)}{n - X(t)} \right)$ \\
Path distance $1$ from $\scr{B}_t$ & $ \frac{2 B(t) }{n} \left( \frac{M_{k_1 -1,k_2 +1}(t)}{n - X(t)} \cdot \bm{1}_{k_1 > 0} - \frac{M_{k_1,k_2}(t)}{n - X(t)} \right) \Gamma(t) - \frac{2B(t)}{n} \frac{A(t)}{2} \frac{C_{k_1,k_2}(t)}{n-X(t)} - \frac{2 B_{k_1,k_2}(t) }{n}$ \\ 
Path distance $1$ from $\scr{M}_t$ & $\frac{2 M(t)}{n} \left( \frac{M_{k_1 -1,k_2 +1}(t)}{n - X(t)} \cdot \bm{1}_{k_1 > 0} - \frac{M_{k_1,k_2}(t)}{n - X(t)} \right) \Gamma(t) - \frac{2M(t)}{n} \frac{A(t)}{2} \frac{C_{k_1,k_2}(t)}{n-X(t)} - \frac{2  M_{k_1,k_2}(t)}{n}$ \\ 
Path distance $1$ from $\scr{R}_t$ & $\frac{2 R(t)}{n } \left( \frac{M_{k_1 -1,k_2 +1}(t)}{n - X(t)} \cdot \bm{1}_{k_1 >0} - \frac{M_{k_1,k_2}(t)}{n - X(t)} - \frac{C_{k_1,k_2}(t)}{n - X(t)} \right) \Gamma(t)$ \\
$\scr{Q}_t$ & $\frac{(X(t)- 5 L(t))}{n} \frac{C_{k_1 -1,k_2}(t)}{D(t)}$ \\
$\scr{R}_t$ & $\frac{R(t)}{n} \frac{C_{k_1,k_2 -1}(t)}{D(t)}$ \\
$\scr{B}_t$ & $-\frac{B_{k_1,k_2}(t)}{n} + \frac{B_{k_1 +1,k_2 -1}(t)}{n} \cdot \bm{1}_{k_2 > 0}$ \\
\hline
\end{tabular}
\end{table}

\begin{proof}[Proof Sketch of Table \ref{table:type_changes_max}]
The explanations for the case of $k_1 + k_2 = q$ are identical to those of $k_1 + k_2 = q-1$, except that vertices of type $(k_1,k_2)$ are created (instead of destroyed) when $u_\tp$ satisfies $u_\tp \in \scr{Q}_t$ or $u_\tp \in \scr{R}_t$.
\end{proof}

We are now ready to inductively prove Lemma~\ref{lem:inductive}. Firstly, when $q=0$, by definition $\tau_{0} =0$, and so $\sigma_{0} := 0$ trivially satisfies the conditions of Lemma~\ref{lem:inductive}. Let us now assume that $q\ge 1$ and for each of $0\le i\le q-1$ we have defined $\sigma_i$ and functions $x,r,y$ and $c_{j,h}$ on $[0, \sigma_{i}]$ for each $j,h \ge 0$ with  $j + h = i$, and Lemma~\ref{lem:inductive} holds for all $0\le i\le q-1$. We shall define $\sigma_q$ which satisfies  $\sigma_{q} > \sigma_{q-1}$, extend each $x, r,y$ and $c_{j,h}$ to $[0, \sigma_q]$, and define new functions $c_{k_1,k_2}$ on $[0,\sigma_q]$ for $k_1 + k_2 =q$. We shall then prove that these functions satisfy the assertion of Lemma~\ref{lem:inductive} with respect to $\tau_q$ and $\sigma_q$, which will complete the proof of the lemma.

Fix a sufficiently small constant $\eps > 0$, and define the bounded domain $\scr{D}_{\eps}$ as the points $(s,x,y,r, (c_{j,h})_{j + h \in \{q-1,q\}})$ such that
$$
\sigma_{q-1} - 1 <s < 3, |x| < 1 - \eps, |y|<2,|r| <2, |c_{j,h}| <2, \eps < \sum_{j,h: \, j+ h =q-1} c_{j,h} < 2.
$$
It will be convenient to define auxiliary functions to simplify our equations below. Specifically, set $b_{k_1,k_2} = k_1 \cdot c_{k_1,k_2}$ and $m_{k_1,k_2} := k_2 \cdot c_{k_1,k_2}$, as well as $b = \sum_{\substack{j,h: \\ j+ h \in \{q-1,q\} }} b_{j,h}$ and $m = \sum_{\substack{j,h: \\ j+ h \in \{q-1,q\}}} m_{j,h}$. Finally, set $\ell = b + m + r$ and $d = \sum_{\substack{j,h: \\ j+ h =q-1}} c_{j,h}$, as well as $\gamma = 1 + y/(1-x)$ and $a = 2y/(1-x)$. Observe the following system of differential equations:
\begin{eqnarray} \label{eqn:greedy_x_and_y}
x' &=& 2 (y + \ell \gamma) \\
y' &=& -2y + 2 \left(1 -x - y\right) - 2 a \ell, 
\end{eqnarray}
and
\begin{eqnarray} \label{eqn:greedy_r}
r'  &=& y \left(\frac{2m}{1-x} -\frac{2r}{1-x} \right) - 2 (b +m) \frac{r \gamma}{1 - x}  \nonumber \\
&&+ \sum_{\substack{j,h: \\ j+ h \in \{q-1,q\}}} 2 b_{j,h} \left( h + \frac{y}{1-x} \frac{m}{1 - x} \right) \nonumber \\
&&+ \sum_{\substack{j,h: \\ j+ h \in \{q-1,q\}}} 2 m_{j,h} \left( h + \frac{y}{1-x} \frac{m}{1 - x} \right)  \nonumber \\
&&-2r \left(1 + \frac{r \gamma}{1 - x} \right) + 2 r \frac{m \gamma}{1-x} - r +O(1/n).
\end{eqnarray}
If $k_1 + k_2 = q-1$, then:
\begin{eqnarray}\label{eqn:greedy_de_min_type}
c_{k_1,k_2}' &=& 
y\left( \frac{2 m_{k_1 -1,k_2 +1}}{1 - x} \cdot \bm{1}_{k_1 >0} -\frac{2 c_{k_1,k_2}}{1 - x} - \frac{2 m_{k_1,k_2}}{1 - x} \right) \nonumber \\
&& + 2 (b +m) \left( \frac{m_{k_1 -1,k_2 +1}}{1 - x} \cdot \bm{1}_{k_1 > 0} - \frac{m_{k_1,k_2}}{1 - x} \right) \gamma \nonumber \\
&& - 2(b +m) \frac{a}{2} \frac{c_{k_1,k_2}}{1-x} - 2 b_{k_1,k_2} - 2  m_{k_1,k_2} \nonumber\\
&& + 2 r \left( \frac{m_{k_1 -1,k_2 +1}}{1 - x} \cdot \bm{1}_{k_1 >0} - \frac{m_{k_1,k_2}}{1 - x} - \frac{c_{k_1,k_2}}{1 - x} \right) \gamma \nonumber \\
&&- (x- 5 \ell) \frac{c_{k_1,k_2}}{d} - r \frac{c_{k_1,k_2}}{d} + b_{k_1 +1,k_2 -1} \cdot \bm{1}_{k_2 > 0} - b_{k_1,k_2}.  
\end{eqnarray}
Otherwise, that is, if $k_1 + k_2 =q$, then:
\begin{eqnarray}\label{eqn:greedy_de_max_type}
c_{k_1,k_2}' &=& 
y\left( \frac{2 m_{k_1 -1,k_2 +1}}{1 - x} \cdot \bm{1}_{k_1 >0} -\frac{2 c_{k_1,k_2}}{1 - x} - \frac{2 m_{k_1,k_2}}{1 - x} \right) \nonumber \\
&& + 2 (b +m) \left( \frac{m_{k_1 -1,k_2 +1}}{1 - x} \cdot \bm{1}_{k_1 > 0} - \frac{m_{k_1,k_2}}{1 - x} \right) \gamma \nonumber \\
&& - 2(b +m) \frac{a}{2} \frac{c_{k_1,k_2}}{1-x} -2 b_{k_1,k_2} - 2  m_{k_1,k_2} \nonumber \\
&& + 2 r \left( \frac{m_{k_1 -1,k_2 +1}}{1 - x} \cdot \bm{1}_{k_1 >0} - \frac{m_{k_1,k_2}}{1 - x} - \frac{c_{k_1,k_2}}{1 - x} \right) \gamma \nonumber \\
&& + (x- 5 \ell) \frac{c_{k_{1} -1,k_2}}{d}  + r \frac{c_{k_1,k_2 -1}}{d} + b_{k_1 +1,k_2 -1} \cdot \bm{1}_{k_2 > 0} - b_{k_1,k_2}.
\end{eqnarray}
The right-hand side of each of the above equations is Lipchitz on the domain $\scr{D}_{\eps}$, as $d$ is bounded below by $\eps$, and $|x| < 1- \eps$. Define
\[
T_{\scr{D}_{\eps}}:=\min\{t \ge 0: (t/n, X(t)/n, Y(t)/n, R(t)/n, (C_{k_1,k_2}(t)/n)_{k_1 + k_2 \in \{q,q-1\}}) \notin \scr{D}_{\eps}\}.
\]
By the inductive assumption, the `Initial Condition' of Theorem~\ref{thm:differential_equation_method} is satisfied for some $\lambda = o(1)$ and values $\sigma_{q-1}, x(\sigma_{q-1}), r(\sigma_{q-1})$ and $c_{j,h}(\sigma_{q-1})$, where $c_{j,h}(\sigma_{q-1}):=0$ for $j +h = q$. Moreover, the `Trend Hypothesis' is satisfied with $\delta = O(1/n)$, by the expected differences of \eqref{eqn:degree_x_change}--\eqref{eqn:degree_c_max_change}. Finally, the `Boundedness Hypothesis' is satisfied with $\beta=O(\log n)$ and $\beta'=O(n^{-1})$ by Lemma~\ref{lem:boundedness_greedy}. Thus, by Theorem~\ref{thm:differential_equation_method}, for every $\xi>0$, a.a.s.\ $X(t)=nx(t/n)+o(n)$, $R(t)=n r(t/n)+o(n), Y(t) = n y(t/n) + o(n)$ and $C_{k_1,k_2}(t)=n c_{k_1,k_2}(t/n)+o(n)$ uniformly for all $\sigma_{q-1} n \le t \le (\sigma(\eps)-\xi) n$, where $x,r,y$ and $c_{k_1,k_2}$ are the unique solution to \eqref{eqn:greedy_x_and_y}--\eqref{eqn:greedy_de_max_type} with the above initial conditions, and $\sigma(\eps)$ is the supremum of $s$ to which the solution can be extended before reaching the boundary of $\scr{D}_{\eps}$. This immediately yields the following lemma.

\begin{lemma}[Concentration of \ref{alg:degree_greedy}'s Random Variables] \label{lem:concentration_random}
For every $\xi > 0$, a.a.s.\ for all  $\tau_{q-1} \le t \le (\sigma(\eps)-\xi) n$
and $k_1, k_2 \ge 0$ such that $k_1 + k_2 \in \{q,q-1\}$,
\[
    \max\{|X(t) -x(t/n) n|,|Y(t) - y(t/n)n|,|R(t) - r(t/n) n|, |C_{k_1,k_2}(t) - c_{k_1,k_2}(t/n) n| \} =  o(n). 
\]
\end{lemma}
As $\scr{D}_{\eps}\subseteq \scr{D}_{\eps'}$ for every $\eps>\eps'$, $\sigma(\eps)$ is monotonically nondecreasing as $\eps\to 0$, and so $\sigma_{q} := \lim_{\eps\to 0+}\sigma(\eps)$ exists. Moreover, the derivatives of the functions $x, r, y$ and $c_{k_1,k_2}$ are uniformly bounded on $(\sigma_{q-1}, \sigma_q)$, as $d = \sum_{\substack{j,h: \\ j+ h =q-1}} c_{j,h}$, so $c_{j,h}/d \le 1$ for $j+h = q-1$. This implies that the functions are uniformly continuous,
and so (uniquely) continuously extendable to $[\sigma_{q-1}, \sigma_q]$. The following limits thus exist:
\begin{eqnarray}
x(\sigma_{q})&:=&\lim_{s\to \sigma_{q}-} x(s) \label{eqn:x_limit} \\ 
y(\sigma_{q})&:=&\lim_{s\to \sigma_{q}-} y(s) \label{eqn:y_limit} \\ 
r(\sigma_{q})&:=&\lim_{s\to \sigma_{q}-} r(s)  \label{eqn:r_limit} \\ 
c_{k_1,k_2}(\sigma_{q})&:=&\lim_{s\to \sigma_{q}-} c_{k_1,k_2}(s). \label{eqn:c_limit}
\end{eqnarray}
The random variables $|R(t)/n|, |Y(t)/n|$ and $|C_{k_1,k_2}(t)/n|$ for $k_1 + k_2 \in \{q,q-1\}$ are bounded by $1$ for all $t$. Thus, when $t/n$ approaches $\sigma_q$, $X(t)/n$ approaches $1$, $t/n$ approaches $3$, or $D(t)/n:= \sum_{\substack{j,h :\\ j + h =q-1}} C_{j,h}(t)/n$ approaches $0$. Formally, we have the following proposition:
\begin{proposition}\label{prop:degree_boundary} For every $\eps>0$, there exists $\xi>0$ such that a.a.s.\ one of the following holds.
\begin{itemize}
    \item $D(t) <  \eps n$ for all $ t\ge (\sigma_q-\xi)n$;
    \item $X(t)>(1-\eps)n$ for all $ t\ge (\sigma_q-\xi)n$;
    \item $\sigma_q=3$.
\end{itemize}
\end{proposition}

The ordinary differential equations \eqref{eqn:greedy_x_and_y}--\eqref{eqn:greedy_de_max_type} again do not have an analytical solution. However, numerical solutions show that $\sigma_q < 3$, and $x(\sigma_q) < 1$.  Thus, after executing \ref{alg:degree_greedy} for $t = \sigma_{q} n + o(n)$ steps, there are $D(t) < \eps n$ vertices of type $q-1$ remaining for some $\eps = o(1)$. At this point, by observing the numerical solution (\ref{eqn:x_limit})--(\ref{eqn:c_limit}) at $\sigma_q$, we know that there exists some absolute constant $0 < p < 1$ such that $(X(t) - 5 L(t))/n \ge p$, where we recall that $L(t)$ counts the total number of coloured vertices at time $t$. Hence, at each step, some vertex of type $q-1$ becomes of type $q$ with probability at least $p$. Thus, by applying Chernoff's bound, one can show that a.a.s.\ after another $O(\eps n/p)=o(n)$ rounds, all vertices of type $q-1$ are destroyed. It follows that a.a.s.\ $|\tau_{q}/n - \sigma_q| =o(1)$, and so Lemma~\ref{lem:inductive} is proven.

\section{Proof of Theorem~\ref{thm:main_lower_bound}}

Suppose that $G_t$ is the graph constructed by a strategy after $t$ rounds, whose edges are $(u_i,v_i)_{i=1}^{t}$. When proving \Cref{thm:main_lower_bound}, it is convenient to refer to the \textit{random vertex} $u_i$ as a {\tt square}, and $v_i$ as a {\tt circle} so that every edge in $G_t$ joins a square with a circle. Recall that for a vertex $x \in [n]$, we say that the square $u_{i}$ \textit{lands} on $x$, or that $x$ is \textit{hit/covered} by $u_{i}$. We extend the analogous terminology to circles.

We begin with the following observations. Suppose that $H$ is a Hamiltonian cycle created in the process. Then,
\begin{enumerate}[label=(\subscript{O}{{\arabic*}})]
    \item $H$ uses exactly $n$ squares;     \label{eqn:exact_squares}
    \item $H$ uses at most $2$ squares on each vertex;  \label{eqn:two_squares}
    \item Suppose $(u_i,v_i)$ is an edge of $G_t$, and $v_i$ received at least two squares. Then, either $H$ uses at most one square on $v_i$, or $H$ does not contain the edge $(u_i,v_i)$.  \label{eqn:disallow_edge}
\end{enumerate}
The first two observations above are straightforward. For \ref{eqn:disallow_edge}, notice that if $H$ uses exactly $2$ squares on $v_i$, then these $2$ squares correspond to 2 edges in $H$ that are incident to $v_i$. Moreover, neither of these edges can be $(u_i,v_i)$, as $u_i$ is the square of $(u_i,v_i)$. Thus, the edge $(u_i,v_i)$ cannot be used by $H$ as $v_i$ has degree 2 in $H$.

Define $Z_x$ as the number of squares on vertex $x \in [n]$. The observation \ref{eqn:two_squares} above indicates the consideration of the random variable 
\[
Z=\sum_{x=1}^n \left( \bm{1}_{Z_x= 1} + 2\cdot \bm{1}_{Z_x\ge 2}\right)
= 2n - \sum_{x=1}^n \left( 2\cdot \bm{1}_{Z_x= 0} + \bm{1}_{Z_x= 1} \right),
\]
which counts the total number of squares that can possibly contribute to $H$, truncated at 2 for each vertex. Observation \ref{eqn:disallow_edge} above indicates the consideration of the following two sets of structures. Let $\W_1$ be the set of pairs of vertices $(x,y)$ at time $t$ such that
\begin{enumerate}
    \item[(a)] $x$ receives its first square at some step $i < t$, and $y$ receives the corresponding circle in the same step;
    \item[(b)] no more squares land on $x$ after step $i$;
    \item[(c)] at least two squares land on $y$ after step $i$.
\end{enumerate}
Let $\W_2$ be the set of pairs of vertices $(x,y)$ at time $t$ such that
\begin{enumerate}
    \item[(a)] $x$ receives a square at some step $i < t$, and $y$ receives the corresponding circle in the same step;
    \item[(b)] $x$ receives exactly two squares (the other square may land on $x$ either before or after step $i$); 
    \item[(c)] at least two squares land on $y$ after step $i$. 
\end{enumerate}

Note that for every $(x,y)\in \W_1$, at most 2 squares on $x$ and $y$ together can be used in $H$, although $x$ and $y$ together contribute $3$ to the value of $Z$. Similarly, for every $(x,y)\in \W_2$, at most 3 squares on $x$ and $y$ together can be used in $H$, although $x$ and $y$ together contribute $4$ to the value of $Z$. We prove the following upper bound on the total number of squares that can possibly contribute to the construction of $H$.

Let 
\begin{eqnarray*}
\T_1&=&\{((x_1,y_1),(x_2,y_2))\in \W_1\times \W_2:\  y_1=x_2\}\\
\T_2&=&\{((x_1,y_1),(x_2,y_2))\in \W_2\times \W_2:\ y_1=x_2\}.
\end{eqnarray*}
Let $W:=|\T_1|+|\T_2|$.

\begin{claim}\label{claim}
    The total number of squares contributing to $H$ is at most $Z-|\W_1|-|\W_2|+W$.
\end{claim}

{\em Proof of Claim~\ref{claim}.} By the discussions above, the total number of squares contributing to $H$ is at most $Z-|\W_1|-|\W_2|+Z'$ where $Z'$ accounts for double counting caused by distinct $(x_1,y_1), (x_2,y_2)\in \W_1\cup \W_2$ where $\{x_1,y_1\}\cap \{x_2,y_2\}\neq \emptyset$. We bound $Z'$ by $W$ by considering the following cases.
    
    {\em Case 1:}
    We first consider the case that $(x_1,y_1),(x_2,y_2)\in \W_1$ where $(x_1,y_1)\neq(x_2,y_2)$ and $\{x_1,y_1\}\cap \{x_2,y_2\}\neq \emptyset$. The only possible situation is $y_1=y_2$. In this case, only 2 squares out of the four squares on $x_1$, $x_2$ and $y_1$ that were counted by $Z$ can contribute to $H$; hence there is no double counting (i.e.\ the offset is -2 as correctly carried out in $-|\W_1|-|\W_2|$).

{\em Case 2:}
    The second case involves $((x_1,y_1),(x_2,y_2))\in \W_1\times \W_2$ where $\{x_1,y_1\}\cap \{x_2,y_2\}\neq \emptyset$. By definition, $x_1\neq x_2$ and $x_1\neq y_2$. We consider the two possible subcases:

{\em Case 2',}
    $x_2=y_1$: In this case, there are 5 squares on the three vertices $x_1$, $y_1=x_2$, and $y_2$ that were counted by $Z$, and at most 4 can contribute to $H$ by using the two squares on $y_1=x_2$ and the two squares on $y_2$. Thus, each such structure causes one double counting and this explains the term $|\T_1|$ in $W$;

{\em Case 2'',}
    $y_1=y_2$: In this case, there are 5 squares on the three vertices $x_1$, $x_2$, and $y_1=y_2$ that were counted by $Z$, and at most 3 can contribute to $H$. Thus, there is no double counting in this subcase.

{\em Case 3:}
    The third case involves $(x_1,y_1),(x_2,y_2)\in \W_2$ where $(x_1,y_1)\neq(x_2,y_2)$ and $\{x_1,y_1\}\cap \{x_2,y_2\}\neq \emptyset$.  We consider the two possible subcases:

{\em Case 3',}
    $y_1=x_2$ (and symmetrically $y_2=x_1$): in this subcase, at most 5 squares out of the 6 squares on the three vertices that were counted by $Z$ can be used to construct $H$ (by using two squres on $x_1$; two squares on $y_2$ and one square on $y_1=x_2$). Note also that for the two pairs of edges $(x_1,y_1)$ and $(x_2,y_2)$ considered here, we may always label them so that $y_1=x_2$. This explains the $|\T_2|$ term in $W$.
 
{\em Case 3'',}
    $y_1=y_2$ or $x_1=x_2$: at most 4 squares out of the 6 squares on the three vertices that were counted by $Z$ can be used to construct $H$; thus there is no double counting in this case.
\qed 

\medskip

The random variable $Z$ is well understood. From the limiting Poisson distribution of the number of squares in a single vertex, we immediately get that, a.a.s., $Z\sim (2-2e^{-s}-e^{-s}s)n$ for $s:=t/n$.

We will estimate the expectation of $|\W_1|, |\W_2|, |\T_1|, |\T_2|$ as well as the concentration of these random variables. However, concentration may fail if the semi-random process uses a strategy which places many circles on a single vertex. Intuitively, placing many circles on a single vertex is not a good strategy for quickly building a Hamiltonian cycle, as it wastes many edges. To formalise this idea, let $\mu:=\sqrt{n}$ (indeed, choosing any $\mu$ such that $\mu\to \infty$ and $\mu=o(n)$ will work). We say that a strategy for the semi-random process is $\mu$-\textit{well-behaved} up until step $t$, if no vertex receives more than $\mu$ circles in the first $t$ steps. In~\cite[Definition 3.2 -- Proposition 3.4]{gao2022perfect}, it was proven that it is sufficient to consider $\mu$-well-behaved strategies in the first $t = O(n)$ steps for establishing a lower bound on the number of steps needed to build a perfect matching. These definitions and proofs can be easily adapted for building Hamilton cycles in an obvious way. We thus omit the details and only give a high-level explanation below.

The key idea is that within $t = O(n)$ steps of any semi-random process, the number of vertices that received more than $\mu$ circles is at most $O(n/\mu)=o(n)$. Therefore, if a Hamiltonian cycle $C$ is built in $t$ steps, then the subgraph $H$ of $C$ induced by the set $S$ of vertices that received at most $\mu$ circles in $G_t$ is a collection of paths spanning all vertices in $S$ which must also contain $n-O(n/\mu)=(1-o(1))n$ edges. We call such a pair $(S,H)$ an \textit{approximate Hamiltonian cycle}. It follows from the above argument that it takes at least as long time to build a Hamiltonian cycle as to build an approximate Hamiltonian cycle. It is then easy to show by a coupling argument that if a strategy builds an approximate Hamiltonian cycle in $t = O(n)$ steps, then there exists a well-behaved strategy that builds an approximate Hamiltonian cycle in $t$ steps as well. Note that observations \ref{eqn:two_squares}--\ref{eqn:disallow_edge} hold for approximate Hamiltonian cycles, and \ref{eqn:exact_squares} holds for approximate Hamiltonian cycles with $n$ replaced by $(1-o(1))n$. Thus,  the following condition has to be satisfied at the time when an approximate Hamiltonian cycle is built: 
\[
Z-|\W_1|-|\W_2|+W\ge (1-o(1))n.
\]
We now estimate the sizes of $\W_1$, $\W_2$, $\T_1$, and $\T_2$ in the semi-random process when executing a well-behaved strategy $\scr{S}$. Crucially, the sizes of these sets do \textit{not} rely on the decisions made by $\scr{S}$. Recall that $(G^{{\mathcal S}}_{s})_{s\ge 0}$ denotes the sequence of graphs produced by ${\mathcal S}$. 

\begin{lemma} \label{lem:f}
Suppose ${\mathcal S}$ is $\mu$-well-behaved. For every $t=\Theta(n)$, a.a.s.\  the following  holds in $G^{{\mathcal S}}_{t}$, 
\[
Z-|\W_1|-|\W_2|+W\sim f(s)n,
\]
where $s:=t/n$ and $f(s)$ is defined as in Theorem~\ref{thm:main_lower_bound}.
\end{lemma}

\begin{proof}
Suppose ${\mathcal S}$ is $\mu$-well-behaved  until time $t = \Theta(n)$. We shall prove that a.a.s.\ the following properties hold for $G^{{\mathcal S}}_{t}$: 
\begin{align}
 |\W_1|&\sim n \sum_{i\le t} \frac{1}{n} \left(1-\frac{1}{n}\right)^{t} \sum_{i< j_1<j_2\le t} \frac{1}{n^2}\left(1-\frac{1}{n}\right)^{j_2}\label{W1} \displaybreak[0] \\
|\W_2|&\sim n \sum_{i_1< i_2\le t} \frac{1}{n^2}\left(1-\frac{1}{n}\right)^{t} \left(\sum_{i_1< j_1<j_2\le t} \frac{1}{n^2}\left(1-\frac{1}{n}\right)^{j_2} + \sum_{i_2<j_1<j_2\le t} \frac{1}{n^2}\left(1-\frac{1}{n}\right)^{j_2}\right) \label{W2} \\
 |\T_1|&\sim n \sum_{i\le t} \frac{1}{n} \left(1-\frac{1}{n}\right)^{t} \sum_{i< j_1<j_2\le t} \frac{1}{n^2}\left(1-\frac{1}{n}\right)^{t}\nonumber\\
 &\times\left(\sum_{j_1< h_1<h_2\le t} \frac{1}{n^2}\left(1-\frac{1}{n}\right)^{h_2} + \sum_{j_2<h_1<h_2\le t} \frac{1}{n^2}\left(1-\frac{1}{n}\right)^{h_2}\right)\label{T1} \\
 |\T_2|&\sim n \sum_{i_1<i_2\le t} \frac{1}{n^2} \left(1-\frac{1}{n}\right)^{t} \sum_{i_1<j_1<j_2\le t} \frac{1}{n^2}\left(1-\frac{1}{n}\right)^{t}\nonumber\\
 &\times\left(\sum_{j_1< h_1<h_2\le t} \frac{1}{n^2}\left(1-\frac{1}{n}\right)^{h_2} + \sum_{j_2<h_1<h_2\le t} \frac{1}{n^2}\left(1-\frac{1}{n}\right)^{h_2}\right)\nonumber\\
 & + \sum_{i_1<i_2\le t} \frac{1}{n^2} \left(1-\frac{1}{n}\right)^{t} \sum_{i_2< j_1<j_2\le t} \frac{1}{n^2}\left(1-\frac{1}{n}\right)^{t}\nonumber\\
 &\times\left(\sum_{j_1< h_1<h_2\le t} \frac{1}{n^2}\left(1-\frac{1}{n}\right)^{h_2} + \sum_{j_2<h_1<h_2\le t} \frac{1}{n^2}\left(1-\frac{1}{n}\right)^{h_2}\right).\label{T2}
\end{align}

We prove~(\ref{W1}) and briefly explain the expressions in~(\ref{W2})--(\ref{T2}) whose proofs are similar to that of~(\ref{W1}). Fix a vertex $x \in[n]$ and a square $u_i$ for $i\le t$. The probability that $u_i$ lands on $x$ in step $i$ is $1/n$. Condition on this event. The probability that $x$ receives no squares in any steps other than $i$ is $(1-1/n)^{t-1}\sim (1-1/n)^{t}$. Let $y$ be the vertex which the strategy chooses to pair $u_i$ with. Fix any two integers  $i<j_1<j_2\le t$, the probability that $y$ receives its first two squares at times $j_1$ and $j_2$ is asymptotically $n^{-2}(1-1/n)^{j_2}$. Summing over all possible values of $i,j_1,j_2$ and multiplying by $n$, the number of choices for $x$, gives $\ex |\W_1|$.

For concentration of $|\W_1|$ we prove that $\ex |\W_1|^2\sim (\ex |\W_1|)^2$. For any pair of $((x_1,y_1),(x_2,y_2))$ in $ \W_1\times \W_1$, either $x_1,y_1,x_2,y_2$ are pairwise distinct, or $y_1=y_2$. It is easy to see that the expected number of pairs where $x_1,y_1,x_2,y_2$ are pairwise distinct is asymptotically
\[
n^2 \sum_{\substack{i_1\le t\\ i_2\le t}} \frac{1}{n^2} \left(1-\frac{1}{n}\right)^{2t} \sum_{\substack{i_1\le j_1<j_2\le t\\
i_2\le h_1<h_2\le t
}} \frac{1}{n^4}\left(1-\frac{1}{n}\right)^{j_2+h_2} \sim (\ex|\W_1|)^2.
\]
The expected number of pairs where $y_1=y_2$ is at most $\mu n$ as there are most $n$ choices for $x_1$ and given $(x_1,y_1)$, there can be at most $\mu$ choices for $(x_2,y_2)$ since ${\mathcal S}$ is $\mu$-well-behaved. Since $\mu=o(n)$, $\mu n=o(n^2)$ which is $o((\ex|\W_1|)^2)$. Thus we have verified that $\ex |\W_1|^2\sim (\ex|\W_1|)^2$ and thus by the second moment method, a.a.s.\ $|\W_1|\sim \ex|\W_1|$. 

The proofs for the expectation and concentration of $|\W_2|$, $|\T_1|$ and $|\T_2|$ are similar. We briefly explain the expressions in~(\ref{W2})--(\ref{T2}):

In~(\ref{W2}), $i_1$ and $i_2$ denote the two steps at which $x$ receives a square. Since there are two squares on $x$, there are two choices of circles, namely $v_{i_1}$ and $v_{i_2}$. The two summations over $(j_1,j_2)$ accounts for the two choices of $v_{i_1}$ and $v_{i_2}$, depending on which is to be covered by two squares. Thus, $j_1$ and $j_2$ denote the steps where the first two squares on $v_{i_1}$ or $v_{i_2}$ arrive. 

In~(\ref{T1}), $i$ denotes the step where $x_1$ receives its only square; $j_1$ and $j_2$ denote the two steps where $y_1=x_2$ receives its two squares. Hence, there are two choices for $y_2$, and $h_1$ and $h_2$ denote the two steps of the first two squares $y_2$ receives.

In~(\ref{T2}), $i_1$ and $i_2$ denote the two steps where $x_1$ receives its two squares---hence there are two choices for $y_1$. Integers $j_1$ and $j_2$ denote the two steps where $y_1=x_1$ receives its two squares---hence there are two choices for $y_2$. Finally, $h_1$ and $h_2$ denote the steps where $y_2$ receives its first two squares. 

By applying the above equations, we deduce that for $t = sn$,
\begin{align*}
\displaybreak[0]
|\W_1|&\sim n e^{-s}\int_{0}^s dx\int_{x}^s dy_1 \int_{y_1}^s e^{-y_2} dy_2 = n e^{-s}\left(1-\frac{e^{-s}s^2}{2}-e^{-s}s-e^{-s}\right) \\
|\W_2|&\sim  ne^{-s}\int_{0}^s d x_1 \int_{x_1}^s d x_2 \left(\int_{x_1}^s d y_1 \int_{y_1}^s e^{-y_2} d y_2 + \int_{x_2}^s d y_1 \int_{y_1}^s e^{-y_2} d y_2  \right)\\
&= n e^{-s} \left(s-e^{-s}s^2 - \frac{e^{-s}s^3}{2} - e^{-s}s\right) \\
|\T_1|&\sim n e^{-2s} \int_{0}^s dx \int_{x}^s dy_1 \int_{y_1}^s d y_2 \left(\int_{y_1}^s d z_1 \int_{z_1}^s e^{-z_2} d z_2 + \int_{y_2}^s d z_1 \int_{z_1}^s e^{-z_2} d z_2 \right) \\
&= n e^{-2s}\left(-1+s - \frac{e^{-s}s^3}{3}-\frac{e^{-s}s^2}{2}-\frac{e^{-s}s^4}{8}+e^s\right)\\
|\T_2|&\sim n e^{-2s} \int_{0}^s dx_1 \int_{x_1}^s dx_2 \int_{x_1}^s dy_1 \int_{y_1}^s d y_2 \left(\int_{y_1}^s d z_1 \int_{z_1}^s e^{-z_2} d z_2 + \int_{y_2}^s d z_1 \int_{z_1}^t e^{-z_2} d z_2 \right)\\
&+n e^{-2s} \int_{0}^s dx_1 \int_{x_1}^s dx_2 \int_{x_2}^s dy_1 \int_{y_1}^s d y_2 \left(\int_{y_1}^s d z_1 \int_{z_1}^s e^{-z_2} d z_2 + \int_{y_2}^s d z_1 \int_{z_1}^s e^{-z_2} d z_2 \right)\\
&= n e^{-2s}\left(-s+s^2-e^{-s}s\left(\frac{s^4}{8}+\frac{s^3}{3}+\frac{s^2}{2}-1\right)\right).
\end{align*}
It follows now that $Z-|\W_1|-|\W_2|+W \sim f(s) n$ where we recall that
\[
f(s)=2+e^{-3s}(s+1)\left(1-\frac{s^2}{2}-\frac{s^3}{3}-\frac{s^4}{8}\right)+e^{-2s}\left(2s+\frac{5s^2}{2}+\frac{s^3}{2}\right)-e^{-s}\left(3+2s\right).
\]
This finishes the proof of the lemma.
\end{proof}

\begin{proof}[Proof of Theorem~\ref{thm:main_lower_bound}]
Recall that $\beta$ is the positive root of $f(s)=1$. Then, for every $\eps>0$, $Z-|\W_1|-|\W_2|+W\le (1-O(\eps)) n$  a.a.s.\ in $G^{{\mathcal S}}_{(\beta-\eps)n}$ for any $\mu$-well-behaved ${\mathcal S}$. Therefore, $C_{\texttt{HAM}}\ge \beta$.
\end{proof}

\section{Conclusion and Open Problems}

We have made significant progress on reducing the gap between the previous best upper and lower bounds on $C_{{\tt HAM}}$. That being said, we do not believe that any of our new bounds are tight. For instance, in the case of our lower bound, one could study the appearance of more complicated substructures which prevent any strategy from building a Hamiltonian cycle. One way to likely improve the upper bound would be to analyze an adaptive algorithm whose decisions are all made greedily. In the terminology of \ref{alg:degree_greedy}, when a (second) square lands on a blue vertex, choose the edge greedily chosen amongst unsaturated vertices of minimum blue degree (as opposed to u.a.r.). Unfortunately, it seems challenging to analyze this algorithm via the differential equation method.

Another direction is to understand which graph properties exhibit \textit{sharp thresholds}. Given property $\scr{P}$, the definition of $C_{\scr{P}}$ ensures that there exists a strategy $\scr{S}^*$ such that for all $\eps > 0$, $G_{t}^{\scr{S}^*}(n)$ satisfies $\scr{P}$ a.a.s.\ for $t \ge (C_{\scr{P}} + \eps)n$. On the other hand, $G_{t}^{\scr{S^*}}(n)$ may satisfy $\scr{P}$ with \textit{constant} probability for $t \le (C_{\scr{P}} - \eps)n$ without contradicting the definition of $C_{\scr{P}}$. For $\scr{P}$ to have a sharp threshold means that, for every strategy $\scr{S}$ and $\eps > 0$, if $t \le (C_{\scr{P}} - \eps)n$ then, a.a.s., $G_{t}^{\scr{S}}(n)$ does \textit{not} satisfy $\scr{P}$.  It is known that for basic properties, such as minimum degree $k \ge 1$, sharp thresholds do exist~\cite{ben2020semi}. Moreover, in \cite{ben-eliezer_fast_2020} it was shown that if $H$ is a spanning graph with maximum degree $\Delta = \omega( \log n)$, then the appearance of $H$ takes $(\Delta/2 + o(\Delta))n$ rounds, and $H$ (deterministically) cannot be constructed in fewer than $\Delta n/2$ rounds. However, in general it remains open as to whether or not a sharp threshold exists when $H$ is \textit{sparse} (i.e., $\Delta = O(\log n)$). Recently, the third author and Surya developed a general machinery for proving the existence of sharp thresholds in adaptive random graph processes~\cite{macrury2022sharp}. Applied to the semi-random graph process, they show that sharp thresholds exist for the property of being Hamiltonian and the property of containing a perfect matching. This provides some evidence that sharp thresholds \textit{do} exist when $\Delta = O(\log n)$, and we leave this as an interesting open problem. 

\bibliographystyle{plain}

\bibliography{refs.bib}

\appendix

\section{The Differential Equation Method}

In this section, we provide a self-contained \textit{non-asymptotic} statement of the differential equation method. The statement combines \cite[Theorem $2$]{warnke2019wormald}, and its extension \cite[Lemma $9$]{warnke2019wormald}, in a form convenient for our purposes, where we modify the notation of~\cite{warnke2019wormald} slightly. In particular, we rewrite \cite[Lemma $9$]{warnke2019wormald} in a less general form in terms of a stopping time $T$. We need only check the `Boundedness Hypothesis' (see below) for $0 \le t \le T$, which is exactly the setting of Lemmas \ref{lem:lipschitz_randomized} and \ref{lem:boundedness_greedy}. A similar theorem is stated in \cite[Theorem $2$]{bennett2023extending}.

Suppose we are given integers $a,n \ge 1$, a bounded domain $\scr{D} \subseteq \mb{R}^{a+1}$, and functions $(F_k)_{1 \le k \le a}$ where each $F_k: \scr{D} \to \mb{R}$ is $L$-Lipschitz-continuous on $\scr{D}$ for $L \ge 0$. Moreover, suppose that $R \in [1, \infty)$ and $S \in (0, \infty)$ are \textit{any} constants which satisfy $\max_{1 \le k \le a} |F_{k}(x)| \le R$ for all $x=(s,y_1,\ldots ,y_{a})\in \scr{D}$ and $0 \le s \le S$.

\begin{theorem}[Differential Equation Method, \cite{warnke2019wormald}] \label{thm:differential_equation_method} 
Suppose we are given $\sigma$-fields $\scr{F}_{0}  \subseteq \scr{F}_{1} \subseteq \cdots$, and for each $t \ge 0$, random variables $((Y_{k}(t))_{1 \le k \le a}$ which are $\scr{F}_t$-measurable. Define $T_{\scr{D}}$ to be the minimum $t \ge 0$ such that
\[
 (t/n, Y_{1}(t)/n, \ldots , Y_{k}(t)/n) \notin \scr{D}.
\]
Let $T \ge 0$ be an (arbitrary) stopping time\footnote{The stopping time $T\ge 0$ is \textit{adapted} to $(\scr{F}_t)_{t \ge 0}$, provided the event $\{\tau = t\}$ is $\scr{F}_t$-measurable for each $t \ge 0$.} adapted to $(\scr{F}_t)_{t \ge 0}$, and assume that the following conditions hold for $\delta, \beta, \gamma \ge 0$ and $\lambda \ge \delta \min\{S, L^{-1}\} + R/n$:
\begin{enumerate}
    \item[(i)] The `Initial Condition': For some $(0,\hat{y}_1,\ldots ,\hat{y}_a) \in \scr{D}$, \label{enum:initial_conditions}
    \[
    \max_{1 \le k \le a} |Y_{k}(0) - \hat{y}_k n| \le \lambda n.
    \] 
    \item[(ii)] The `Trend Hypothesis': For each  $t \le \min\{ T, T_{\scr{D}} -1\}$, \label{enum:trend_hypothesis}
    $$|\mb{E}[ Y_{k}(\tp) - Y_{k}(t) \mid \scr{F}_t] - F_{k}(t/n,Y_{1}(t)/n,\ldots ,Y_{a}(t)/n)| \le \delta.$$
    \item[(iii)] The `Boundedness Hypothesis': With probability $1 - \gamma$, \label{enum:boundedness_hypothesis}
    $$|Y_{k}(\tp) -  Y_{k}(t)| \le \beta,$$
    for each $t \le \min\{ T, T_{\scr{D}} -1\}$.
\end{enumerate}
Then, with probability at least $1 - 2a \exp\left(\frac{-n \lambda^2}{8 S \beta^2}\right) - \gamma$, we have that
\begin{equation}
    \max_{0 \le t \le \min\{T, \sigma n\}} \max_{1 \le k \le a} |Y_{k}(t) -y_{k}(t/n) n| < 3 \lambda \exp(L S)n,
\end{equation}
where $(y_{k}(s))_{1 \le k \le a}$ is the unique solution to the system of differential equations
\begin{equation} \label{eqn:general_de_system}
    y_{k}'(s) = F_{k}(s, y_{1}(s),\ldots ,y_{a}(s)) \quad \mbox{with $y_{k}(0) = \hat{y}_k$ for $1 \le k \le a$,}
\end{equation}
and $\sigma = \sigma(\hat{y}_1,\ldots ,\hat{y}_a) \in [0,S]$ is any choice of $\sigma \ge 0$ with the property that $(s,y_{1}(s),\ldots, y_{a}(s))$ has $\ell^{\infty}$-distance at least $3 \lambda \exp(LS)$ from the boundary of $\scr{D}$ for all $s \in [0, \sigma)$.
\end{theorem}

\begin{remark}
Standard results for differential equations guarantee that \eqref{eqn:general_de_system} has a unique solution $(y_{k}(s))_{1\le k \le a}$ which extends arbitrarily close to the boundary of $\scr{D}$. 
\end{remark}

\end{document}